\documentclass{amsart}

\usepackage[margin=1.5in]{geometry}
\usepackage{latexsym,amsxtra,amscd,ifthen,amsmath,color, amsthm}
\usepackage{amsfonts}

\makeatletter
\@namedef{subjclassname@2020}{\textup{2020} Mathematics Subject Classification}
\makeatother

\usepackage{verbatim}
\usepackage{hyperref}
\usepackage{amsthm, tikz, tikz-cd, bbm}
\usepackage{amssymb}
\usepackage[all,cmtip]{xy}
\usepackage[normalem]{ulem} 
\usepackage[shortlabels]{enumitem}
\usepackage{xcolor}

\definecolor{darkgreen}{RGB}{55,138,0}
\definecolor{burntorange}{RGB}{180,85,0}
\definecolor{navyblue}{RGB}{18,40,180}
\definecolor{cyan(process)}{rgb}{0.0, 0.6, 1.0}
\definecolor{blue-violet}{rgb}{0.54, 0.17, 0.89}

\hypersetup{
	colorlinks  = true,  
	urlcolor   = blue,  
	linkcolor  = navyblue,  
	citecolor  = black   
}
\allowdisplaybreaks

\numberwithin{equation}{section}

\theoremstyle{plain}
\newtheorem{theorem}{Theorem}[section]
\newtheorem*{theorem*}{Theorem}
\newtheorem*{definition*}{Definition}
\newtheorem{lemma}[theorem]{Lemma}

\newtheorem{proposition}[theorem]{Proposition}

\theoremstyle{definition}
\newtheorem{definition}[theorem]{Definition}
\newtheorem{example}[theorem]{Example}
\newtheorem{remark}[theorem]{Remark}

\newcommand{\Z}{\mathbb{Z}}
\newcommand{\N}{\mathbb{N}}
\newcommand{\kk}{\Bbbk}
\newcommand{\Id}{\textsl{Id}}
\newcommand{\HH}{\ensuremath{\mathsf{HH}}}
\newcommand{\Ext}{\operatorname{Ext}}
\newcommand{\rad}{\operatorname{rad}}
\newcommand{\Hom}{\operatorname{Hom}}

\newlist{steps}{enumerate}{1}
\setlist[steps, 1]{label = Step \arabic*:}

\begin{document}
	
	\title[]{Maurer-Cartan equation for gentle algebras}
	
	\author[M. M\"{u}ller]{Monique M\"{u}ller $^1$}
	
	\author[M. J. Redondo]{Mar\'\i a Julia Redondo $^{2, 3}$}

	\author[F. Rossi Bertone]{Fiorela Rossi Bertone $^2$}
	
	\author[P. Suarez]{Pamela Suarez $^4$}
	
	\address{$^1$ Departamento de Matem\'atica e Estat\'istica, Universidade Federal de S\~ao Jo\~ao del-Rei, Pra\c ca Frei Orlando, 170, Centro, S\~ao Jo\~ao del-Rei, Minas Gerais, Brazil, CEP: 36307-352 }
	\email{monique@ufsj.edu.br}

	\address{$^2$ Instituto de Matem\'atica (INMABB), Departamento de Matem\'atica, Universidad Nacional del Sur (UNS)-CONICET, Bah\'\i a Blanca, Argentina}
	\email{mredondo@uns.edu.ar,  fiorela.rossi@uns.edu.ar}
	
	\address{$^3$ Guangdong Technion Israel Institute of Technology, Shantou, Guangdong Province, China}
	\address{$^4$ Centro marplatense de Investigaciones Matem\'aticas, Facultad de Ciencias Exactas y
		Naturales,  Universidad Nacional de Mar del Plata,  Mar del
		Plata, Argentina}
	\email{psuarez@mdp.edu.ar}

	\begin{abstract} 
		Let $A= \kk Q/I$ be a finite dimensional gentle algebra. In this article, under some hypothesis on the quiver $Q$, we give conditions for nilpotency of the $L_\infty$-structure on the shifted Bardzell's complex $B(A)[1]$. For nilpotent cases, we describe Maurer-Cartan elements.
	\end{abstract}
	
	\subjclass[2020]{16S80;16E40;18G35}
	\keywords{Maurer-Cartan elements, gentle algebras, deformations.}

	\maketitle

	\setcounter{tocdepth}{2}

	\section{Introduction}

	Given an associative algebra $A$, it is well known that the set of equivalence classes of deformations of $A$, in the sense of Gerstenhaber \cite{G1}, is in one-to-one correspondence with the set of Maurer-Cartan elements modulo Gauge equivalence, see for instance \cite{Man}. 
	In order to find Maurer-Cartan elements, one can use the shifted Hochschild complex $C(A)[1]$ which, endowed with the Gerstenhaber bracket, admits a structure of dg-Lie algebra.

	The natural generalization of dg-Lie algebras is the concept of $L_\infty$-algebras. Moreover, two different deformation problems are equivalent if the corresponding dg-Lie algebras are equivalent as $L_\infty$-algebras.

	When $A$ is a  monomial algebra, Bardzell's complex $B(A)$ has shown to be more efficient when dealing with computations of the Hochschild cohomology groups.  Comparison morphisms between $C(A)$ and $B(A)$ have been described explicitly in \cite{RR1}.  In \cite{RRB} the authors explicitly translate the dg-Lie algebra structure from $C(A)[1]$ to $B(A)[1]$ using the existence of a contraction. Hence, Maurer-Cartan elements can be computed using the generalized Maurer-Cartan equation
	$$l_1 (f) - \sum_{n \geq 2} (-1)^{\frac{(n+1)n}{2}} \frac{1}{n!} l_n(f, \dots , f) =0,$$
	for $f = \sum_{i \geq 1} f_i t^i$ with $f_i \in B^1(A)[1]$. 
	In general, this equation contains  an infinite sum and its convergence may be guaranteed by some nilpotence condition, see for instance \cite{Get, Yalin}. We are interested in finding conditions on $A$ such that $l_n(f_1, \dots , f_n) =0$ for all $f_i \in B^1(A)[1]$ for $n\gg 0$.

	Gentle algebras constitute a large and well-studied class of quadratic monomial algebras. They first appeared as iterated tilted algebras of Dynkin type $A_n$ and $\tilde{A}_n$, see \cite{AH, AS}, so they can be seen as generalizations of algebras of Dynkin type $A_n$.
	
	In this paper, we deal with gentle algebras $A=\kk Q/I$ whose quiver $Q$ has no parallel arrows and no oriented cycles. For them, in Theorem \ref{teo ln=0}, we give conditions in the quiver to ensure the nilpotency of the brackets $l_n(f_1,\dots,f_n)$ for elements $f_i\in B^1(A)[1]$. Moreover, we find a sufficient condition for the vanishing of
	$l_n(f_1,\dots,f_n)$ for $n\geq 5$, and a necessary condition for the nonvanishing of $l_4(f_1, f_2,f_3,f_4)$, for all $f_i\in B^1(A)[1]$. Furthermore, for nilpotent brackets, we describe the Maurer-Cartan elements in Theorem \ref{teo MC elements} and prove that, in fact,  they are just given by $2$-cocycles.

	The paper is organized as follows. Section \ref{sec 2} contains definitions and notation about gentle algebras, Hochschild and Bardzell's complexes, $L_\infty$-algebras, and Maurer-Cartan elements. 
	In Section \ref{sec 3} we give the needed formulae for quadratic algebras. Section \ref{sec 4} is devoted to the study of gentle algebras and the description of the subquivers that may imply the nonvanishing of the brackets $l_n(f_1,\dots,f_n)$ for elements $f_i\in B^1(A)[1]$. Finally, in Section \ref{sec 5} we present our main results concerning Maurer-Cartan elements.

	\section{Preliminaries}\label{sec 2}
	Let $\kk$ be a field of characteristic zero. 
	
	\subsection{Quivers, relations and gentle algebras} 
	Consider an associative finite dimensional algebra  $A$
	which is a quotient of a path algebra $\kk Q/I$, for a quiver $Q$ and an admissible ideal $I$. We write $Q_0$ and $Q_1$ for the set of vertices and arrows, respectively. Also, $s, t: Q_1\to Q_0$ denote the source and target maps of an arrow. A path $w = \alpha_1 \cdots \alpha_n$ in $Q$ of length $n \geq 1$ is a sequence of arrows $\alpha_1, \ldots, \alpha_n$
	such that $t(\alpha_i) = s(\alpha_{i+1})$ for $1 \leq i < n$. Two paths in $\kk Q$ are said to be parallel if they share starting and ending points. The ideal $I$ is generated by a set $\mathcal{R}$ of paths that are minimal with respect to inclusion of paths, and $E=\kk Q_0$.
	If $I$ is generated by paths of length two, the algebra is called quadratic. 
	By abuse of notation we use the same letters to refer to paths in the path algebra $\kk Q$ or the quotient algebra $A$.
	
	\begin{definition}
		A quadratic algebra  $A=\kk Q/I$ is gentle if it satisfies the following conditions:
		\begin{enumerate}
			\item  $Q$ is finite;
			\item If $\alpha \in Q_1$, then  there exists at most one arrow $\beta$ with $s(\beta)=t(\alpha)$  such that $\alpha \beta \in I$ and
			at most one arrow $\gamma$ with $t(\gamma) = s(\alpha)$ such that $\gamma \alpha \in I$;
			\item If $\alpha \in Q_1$, then there exists at most one arrow $\beta$ such that $\alpha \beta \not \in I$ and at most one arrow $\gamma$ such that $\gamma \alpha \not \in I$;
			\item The admissible ideal $I$ is generated by the relations in (b).
		\end{enumerate} 
	\end{definition}
	
	Notice that from (b) and (c) we have that there are at most two incoming and
	at most two outgoing arrows at each vertex in the quiver $Q$.

	\subsection{Hochschild and Bardzell's complexes}
	
	The  \textit{bar resolution} of $A$ 
	$$C_*(A)= (A \otimes A^{\otimes n} \otimes A, d_n)_{n \geq 0}$$
	is an standard free resolution of the $A$-bimodule $A$ where the differential is given by
	\begin{align*}
		d_n (a_0 \otimes \cdots \otimes  a_{n+1}) =  \sum_{i=0}^{n} (-1)^{i}  a_0 \otimes \cdots \otimes a_{i-1} \otimes & a_i a_{i+1} \otimes  a_{i+2} \otimes \cdots \otimes a_{n+1}.
	\end{align*}
	The $\kk$-$A$-map
	\begin{equation*} 
		s_n :A \otimes A^{\otimes n} \otimes  A \to A \otimes A^{\otimes n+1} \otimes A
	\end{equation*}
	defined by $s_n(x) = 1 \otimes x$ for any $x \in  A \otimes A^{\otimes n} \otimes  A$ is an homotopy contraction, that is, 
	\begin{equation*}
		\Id = s_{n-1}d_n + d_{n+1} s_n.
	\end{equation*}
	The  \textit{Hochschild complex} $C^*(A)= ( \Hom_\kk (A^{\otimes n}, A),  d^n)_{n \geq 0}$ is obtained by applying the functor $\Hom_{A-A}(-,A)$ to the bar resolution, and using the isomorphism
	\begin{align*}
		\Hom_{A-A}(A\otimes V \otimes A,A) & \simeq \Hom_\kk (V,A), \qquad
		\hat f \mapsto  f 
	\end{align*}
	given by $f (v) = \hat f (1 \otimes v \otimes 1)$.
	The differential of a $n$-cochain is the $(n+1)$-cochain given by
	\begin{align*}
		&(-1)^n (d^nf)(a_0 \otimes  \cdots \otimes a_{n}) =   \hat f d_{n+1} (1 \otimes a_0 \otimes  \cdots \otimes a_{n} \otimes 1)  =  a_0 f(a_1 \otimes \cdots \otimes a_n) \\
		&\, - \sum_{i=0}^{n-1} (-1)^{i} f(a_0 \otimes \cdots \otimes a_{i-1}\otimes a_i a_{i+1} \otimes a_{i+1} \otimes \cdots \otimes a_n) 
		+ (-1)^{n+1} f(a_0 \otimes \cdots \otimes a_{n-1})a_n.
	\end{align*}
	Its cohomology is the  \textit{Hochschild cohomology} $\HH(A)$ of $A$ with coefficients in $A$. 
	
	It is well known that the Hochschild cohomology groups $\HH^n(A)$ can be identified with the groups $\Ext^n_{A-A}(A,A)$  since $A$ is $\kk$-projective. Thus, we can replace the bar resolution with any convenient resolution of the $A$-bimodule $A$.
	When $A$ is a monomial algebra,  the Bardzell's resolution, see \cite{B}, have been mainly used to compute the Hochschild cohomology. This resolution  is defined as
	$$B^*(A) = (\Hom_{E-E} ({\kk} AP_n, A),  (-1)^n \delta^n)_{n \geq 0}$$
	where $AP_n$ is the set of supports of $n$-concatenations  associated to a presentation $(Q,I)$ of $A$, $E=\kk Q_0$ and $\delta^n (f) = \hat f \delta_{n+1}$,  see \cite[Section 2.3]{RR1} for precise definitions.

	\subsection{Gerstenhaber bracket}
	
	Let $f\in \Hom_\kk (A^{\otimes n}, A)$  and $g\in \Hom_\kk (A^{\otimes m}, A)$. The Gerstenhaber bracket $[f, g] \in \Hom_\kk (A^{\otimes n+m-1}, A),$ is given by
	$$[f, g] = f \circ g -  (-1)^{(n-1)(m-1)} g \circ f $$
	where $\circ$ denotes the Gerstenhaber product
	\begin{align*}
		f \circ g= \sum_{i=0}^{n-1} (-1)^{i(m+1)} f \circ_i g =\sum_{i=0}^{n-1} (-1)^{i(m+1)} f (\Id^{\otimes i} \otimes g \otimes \Id^{\otimes n-i-1}).
	\end{align*}
	In particular, for $n=m=2$,
	\begin{align*}
		[f, g] =  f(g \otimes \Id - \Id \otimes g) +  g(f \otimes \Id - \Id \otimes f).
	\end{align*}
	It is well known that the shifted Hochschild complex
	$C^*(A)[1]$ endowed with the Gerstenhaber bracket is a dg-Lie algebra, see \cite{Tam, TT}.
	
	\subsection{\texorpdfstring{$L_{\infty}$}{Linf}-structure}\label{subsec l inf} Following \cite{RRB}, we recall the structure of $L_{\infty}$-algebra on the shifted Bardzell's complex and the weak $L_{\infty}$-equivalence with the shifted Hochschild's complex.
	
	\begin{definition}
		A $L_\infty$-algebra is a $\Z$-graded vector space $L$ together with linear maps
		$$l_n: \otimes^n L \to L$$
		of degree $2-n$ subject to the following axioms:
		\begin{itemize}
			\item for every $n \in \N$, every homogeneous $v_1, \ldots, v_n \in L$, and every $\sigma \in \mathbb{S}_n$,
			$$l_n (v_{\sigma(1)},\ldots,  v_{\sigma(n)}) = \chi(\sigma)\  l_n(v_1, \dots,  v_n);$$
			\item for every $n \in \N$, and every homogeneous $v_1, \ldots, v_n \in L$,
			$$\sum_{i+j=n+1}  \sum_{\sigma \in \mathbb{S}_{i,n-i}} (-1)^{i {(j-1)}}  \chi(\sigma)  \
			l_j ( l_i (v_{\sigma(1)},  \ldots, v_{\sigma(i)}), v_{\sigma(i+1)},\ldots, v_{\sigma(n)})  =0
			$$
			where $\chi(\sigma)$ denotes the antisymmetric Koszul sign.
		\end{itemize}	
	\end{definition}
	Let $B=B^*(A)[1]$ and $C=C^*(A)[1]$ be the shifted Bardzell's and Hochschild's complexes, respectively. Let $F_*, G_*$ be the comparison morphisms between the corresponding projective resolutions described explicitly in \cite{RR1} and in \cite{BW}. Let  $H_*$ be the homotopy map between $\Id$ and $F_*G_*$  given by the $A$-$A$-maps $$H_n: C_n(A) \to C_{n+1}(A)$$ defined by $H_0=0$ and  $$H_n= \Id \otimes
	F_n G_n s_{n-1}   - \Id \otimes  H_{n-1} d_n s_{n-1}.$$
	Applying the functor $\Hom_{A-A}(-, A)$ we get the functors $F^*, G^*, H^*$, and we define recursively the linear maps 
	$$l_n: \otimes^n B\to B, \, v_n: \otimes^n B\to C, \mbox{ and } \phi_n: \otimes^n B\to C$$ 
	by $l_1= (-1)^{*+1}\delta^*$, $v_1 =  0$, $\phi_1= G^*$;
	and, for $n \geq 2$ and homogeneous $f_1, \dots, f_n \in B$, 
	\begin{align}\label{eq: vn}
		v_n  & =  \sum_{t=1}^{n-1} \sum_{\tau \in \mathbb{S}^{-}_{t,n-t}} \chi(\tau){ \kappa (\tau)_{t}} \  [ \phi_t , \phi_{n-t}] \hat \tau  \\  \notag
		l_n  & =   F^* v_n, \\ 
		\notag
		\phi_n & = H^*  v_n
	\end{align}
	where
	\begin{align*}
		\hat \tau (f_1 \otimes  \cdots \otimes f_n) & = f_{\tau(1)} \otimes \cdots \otimes f_{\tau(n)}, \mbox{ and}  \\
		[ \phi_t , \phi_{n-t}] (f_1, \ldots, f_n) & = [ \phi_t (f_1, \ldots, f_t), \phi_{n-t}(f_{t+1}, \ldots, f_n)].
	\end{align*}
	
	\begin{theorem} \cite{RRB}\label{infinito-bardzell} With the notation above, the maps  $l_n:\otimes^n B\to B$, $n\in \N$,  give $B$  a $L_\infty$-structure, and the quasi-isomorphism $G$ extends to a weak $L_\infty$-equivalence $\phi: B \to C$.
	\end{theorem}
	
	\subsection{Maurer-Cartan elements}
	
	Let $L$ be a $L_\infty$-algebra. The set $\mathcal{MC}(L)$ of Maurer-Cartan elements consists of all $f \in L^1$ satisfying the generalized Maurer-Cartan equation
	$$l_1 (f) - \sum_{n \geq 2} (-1)^{\frac{(n+1)n}{2}} \frac{1}{n!} {l_n(f, \ldots, f)} =0.$$
	It is well known that the formal deformations of an algebra $A$ over $\Bbbk [[t]]$ are in one-to-one correspondence with equivalence classes of Maurer-Cartan elements in $\mathcal{MC}( C^*(A)[1] \otimes ((t)))$, see for instance \cite[\S 5]{DMZ}.
	
	When $A$ is a monomial algebra and $$B^*(A)[1]= (B^{n+1}(A), {(-1)^{n}} \delta^{n+1}, l_n) $$
	is the $L_\infty$-algebra defined in Subsection \ref{subsec l inf},   Theorem \ref{infinito-bardzell}  and  \cite[Theorem 8.13]{Y}  implies that 
	$$\mathcal{MC} (\overline C^*(A)[1] \otimes ((t))) \simeq \mathcal{MC} (B^*(A)[1] \otimes ((t)))$$
	and  $f = \sum_{i \geq 1} f_i t^i$ with $f_i \in B^1(A)[1]$ satisfies the 
	generalized Maurer-Cartan equation if and only if
	\begin{equation} \label{MC}
		-  \delta^2(f_i) - \sum_{n \geq 2}  \ \sum_{j_1 + \cdots + j_n=i} (-1)^{\frac{(n+1)n}{2}} \frac{1}{n!} {l_n (f_{j_1} , \ldots , f_{j_n})} =0.\end{equation}

	\section{Formulae for quadratic algebras}\label{sec 3}
	
	From now on, unless explicitly stated, all paths will be considered in $A$, not in $Q$. 
	For any quadratic algebra $A=\kk Q/I$  we present all the needed formulae that will be used throughout this article.
	Recall that $AP_0= Q_0, AP_1= Q_1$, and
	\[ AP_n= \{ \omega=(\alpha_1,  \ldots, \alpha_n): \alpha_i \alpha_{i+1} =0 \mbox{ for every $i$} \}.\]
	The $A^e$-linear maps $F_n : A\otimes \kk AP_n \otimes A \to A^{\otimes n+2}$ and $G_n : A^{\otimes n+2} \to A\otimes \kk AP_n \otimes A $ are defined by
	\begin{align*}
		&F_0(1\otimes e\otimes 1)=e\otimes 1,  \qquad F_1(1\otimes \alpha\otimes 1)= 1\otimes \alpha\otimes 1, \\
		&G_0(1\otimes  1)=1\otimes 1 \otimes 1, \qquad G_1(1\otimes \beta_1\cdots \beta_s \otimes 1)= \sum_{i=1}^s \beta_1\cdots\beta_{i-1}\otimes \beta_i\otimes \beta_{i+1}\cdots\beta_s,
	\end{align*}
	if $\alpha \in Q_1$, $\beta_1\cdots \beta_s$ a nonzero path in $A$, and 
	\begin{align*}
		& F_n(1\otimes (\alpha_1,  \ldots, \alpha_n) \otimes 1) 
		= 1 \otimes \alpha_1 \otimes \cdots \otimes \alpha_n \otimes 1  \\
		& G_n(1\otimes u \alpha_1\otimes \alpha_2 \otimes \cdots \otimes \alpha_{n-1} \otimes \alpha_n v \otimes 1)= 
		u \otimes (\alpha_1, \ldots,   \alpha_n)
		\otimes v\end{align*}
	if $\alpha_i\alpha_{i+1} =0$ for any $i$, $u \alpha_1, \alpha_n v$  nonzero paths in $A$, and it is zero otherwise.
	For $n=1,2$, the $A^e$-map $H_n: C_n(A) \to C_{n+1}(A)$ is given by
	\begin{align*}
		H_1 (1 \otimes  \beta_1 \cdots \beta_m \otimes 1) & = \sum_{i=2}^m 1 \otimes \beta_1 \cdots \beta_{i-1} \otimes \beta_i \otimes \beta_{i+1} \cdots \beta_m
	\end{align*}
	\begin{align*}
		H_2 (1 \otimes  \beta_1 \cdots \beta_m \otimes \gamma_1\cdots \gamma_r \otimes 1) =& 
		\ \Theta(\beta_m, \gamma_1)  \
		1 \otimes \beta_1 \cdots \beta_{m-1} \otimes \beta_m \otimes \gamma_1 \otimes \gamma_2 \cdots \gamma_r 
		\\
		& - \sum_{i=2}^s 1 \otimes \beta_1 \cdots \beta_{m} \otimes  \gamma_1  \cdots \gamma_{i-1} \otimes \gamma_i \otimes \gamma_{i+1} \cdots \gamma_r.
	\end{align*}
	with $\Theta(\beta_m, \gamma_1)=1$ if $\beta_m \gamma_1 =0$, and zero otherwise.
	
	A basis of $B^1(A)[1] = \Hom_{E-E}(\kk AP_2,A)$ is given by the set of
	$\kk$-linear maps $(\alpha_1\alpha_2||u)$ with $\alpha_1\alpha_2=0$, $u$ a nonzero path parallel to $\alpha_1\alpha_2$, and 
	\[
	(\alpha_1\alpha_2||u)(\omega) = 
	\begin{cases}
		u, & \mbox{if $\omega  = (\alpha_1,\alpha_2)$};\\
		0, & \mbox{otherwise.} 
	\end{cases}\]
	
	Since we are interested in computing Maurer-Cartan elements, we have to study the behaviour of 
	\[ l_n(f_1, \dots, f_n)(\omega)\]
	for any $f_i \in B^{1}(A)[1] $ 
	and $\omega= (\alpha_1, \alpha_2, \alpha_3) \in AP_3$. For $n=1$ we 
	have 
	\begin{equation*}
		l_1(f)(\omega)= - \delta^2(f)(\alpha_1,\alpha_2, \alpha_3 )= f( \alpha_1,\alpha_2) \alpha_3-\alpha_1f( \alpha_2,\alpha_3),
	\end{equation*} 
	and, for any $n \geq 2$,
	\begin{align*}
		l_n(f_1,\ldots, f_n)(\omega)  &=v_n(f_1, \ldots, f_n)F_3(1 \otimes (\alpha_1,\alpha_2,\alpha_3) \otimes 1)\\ 
		&=v_n(f_1, \ldots, f_n)(1\otimes \alpha_1\otimes \alpha_2\otimes \alpha_3\otimes 1).
	\end{align*} 
	Hence, using \eqref{eq: vn} and the fact that
	\[\phi_k (f_{i_1} , \dots , f_{i_k})(1 \otimes \alpha_i \otimes \alpha_{j} \otimes 1) = v_k (f_{i_1} , \dots , f_{i_k}) H_2 (1 \otimes \alpha_i \otimes \alpha_{j}\otimes 1) =0\]
	for any $k >1$, we obtain that 
	\begin{align} \label{l_n}
		l_n(f_1,\dots, f_n)(\omega)  
		= &(-1)^n\sum\limits_{i=1}^n \phi_{n-1}(f_1,\ldots, \hat{f}_i, \ldots f_n)(1\otimes f_i(\alpha_1\alpha_2)\otimes \alpha_3\otimes 1)\\ &- (-1)^n\sum\limits_{i=1}^n \phi_{n-1}(f_1,\ldots, \hat{f}_i, \dots f_n)(1\otimes \alpha_1\otimes f_i(\alpha_2\alpha_3)\otimes 1). \notag \end{align}

	The following lemma is a direct generalization of \cite[Lemma 4.5]{RRB}.
	
	\begin{lemma}\label{lemma:phi=0-RBR}  Let $A=\kk Q/I$ be a quadratic algebra,  $u, v$ paths in $\rad A$ such that $uv \neq 0$.
		Then
		$$\phi_n (f_1, \dots, f_n) (1 \otimes u \otimes v \otimes 1)=0$$  
		for all $n \geq 1$ and for any $f_i \in B^{1}(A)[1]$.
	\end{lemma}
	
	\begin{proof}  
		The statement is clear if $n=1$ since $\phi_1(f)=\hat{f} \circ G_2$ and $G_2(1 \otimes u \otimes v \otimes 1)=0$. Now we proceed by induction. Let $u= \beta_1 \cdots \beta_m$ and $v=\gamma_1 \cdots \gamma_r$. For $n \geq 2$, 
		$$\phi_{n}(f_1,\ldots,f_n)(1\otimes  u\otimes v \otimes 1)=v_{n}(f_1,\ldots,f_n)H_2(1\otimes  u\otimes v \otimes 1)$$
		and 
		\[H_2 (1 \otimes  \beta_1 \cdots \beta_m \otimes \gamma_1\cdots \gamma_r \otimes 1)  = - \sum_{i=2}^s 1 \otimes \beta_1 \cdots \beta_{m} \otimes  \gamma_1  \cdots \gamma_{i-1} \otimes \gamma_i \otimes \gamma_{i+1} \cdots \gamma_r.\]
		The result follows since $v_n   =  \sum_{t=1}^{n-1} \sum_{\tau \in \mathbb{S}^{-}_{t,n-t}} \chi(\tau){ \kappa (\tau)_{t}} \  [ \phi_t , \phi_{n-t}] \hat \tau$ with
		$$ [ \phi_t , \phi_{n-t}] = \phi_t (\phi_{n-t} \otimes \Id - \Id \otimes \phi_{n-t}) +\phi_{n-t} (\phi_{t} \otimes \Id - \Id \otimes \phi_{t})$$ and, for any $k$ with $1 \leq k <n$,
		\begin{align*}
			\phi_k (f_{i_1}, \ldots, f_{i_k}) (1 \otimes \beta_1 \cdots \beta_{m} \otimes  \gamma_1  \cdots \gamma_{i-1} \otimes 1) & =0, \mbox{ and} \\
			\phi_k (f_{i_1}, \ldots, f_{i_k}) (1 \otimes  \gamma_1  \cdots \gamma_{i-1} \otimes \gamma_i \otimes 1) &= 0
		\end{align*}
		by inductive hypothesis. 
	\end{proof}
	
	\begin{remark}\label{rmk 4.6} 
		Let $f_1,  \dots, f_n \in B^{1}(A)[1]$, 
		$\alpha, \beta \in Q_1$, $v \in \rad A$.   If $v\alpha\neq 0$, we have that
		\begin{align*}
			v_n (f_1, \dots, f_n)(1 \otimes v \otimes \alpha  \otimes & \beta \otimes 1)= \\
			& (-1)^{n+1} \sum \phi_{n-1}(f_1,\dots,\hat{f}_i,\dots,f_n)(1\otimes v\otimes f_i(\alpha\beta)\otimes 1)
		\end{align*}
		since
		$\phi_t(f_{i_1},\dots,f_{i_t})(1\otimes v \otimes \alpha \otimes 1)=0$ for all $t\geq 1$ by Lemma \ref{lemma:phi=0-RBR} and $\phi_t(f_{i_1},\dots,f_{i_t})(1\otimes \alpha\otimes \beta\otimes 1) = v_t (f_{i_1},\dots,f_{i_t}) H_2(1\otimes \alpha\otimes \beta\otimes 1) =0$ for $t\geq 2$.  Then, arguing as in the proof of \cite[Proposition 4.6]{RRB}, the result follows.
	\end{remark}

	\begin{lemma} \label{lemma: no parallel} Let $A=\kk Q/I$ be a quadratic algebra, 
		$u=\beta_1\cdots\beta_m$,  $v= \gamma_1\cdots \gamma_r$ with $\beta_m \gamma_1 =0$, and $f_1,\dots, f_n \in B^{1}(A)[1]$.
		If, for all $i$, there is no nonzero path $w$ appearing as a summand of $f_i(\beta_m \gamma_1)$
		such that $\beta_{m-1} w=0$ or $w\gamma_2=0$,
		then
		\[\phi_n(f_1, \ldots, f_n)(1\otimes u \otimes v \otimes 1)=0\]
		for all $n\geq 2$.
	\end{lemma}
	
	\begin{proof}  For any $n\geq 2$, by definition, $\phi_n(f_1, \dots, f_n)(1\otimes u \otimes v \otimes 1)$ equals
		\begin{align*}
			& v_n (f_1, \ldots, f_n) (1 \otimes \beta_1 \cdots \beta_{m-1} \otimes \beta_m \otimes \gamma_1 \otimes 1) \gamma_2 \cdots \gamma_r \\
			& - \sum_{i=2}^r v_n (f_1, \ldots, f_n) (1\otimes \beta_1 \cdots \beta_{m} \otimes  \gamma_1  \cdots \gamma_{i-1} \otimes \gamma_i \otimes 1)\gamma_{i+1} \cdots \gamma_r.
		\end{align*}
		We proceed by induction on $r$.  If $r=1$, from Lemma \ref{lemma:phi=0-RBR}, for any $t \geq 1$ we get that 
		$\phi_t (f_{i_1}, \ldots, f_{i_t})(1 \otimes \beta_1 \cdots \beta_{m-1} \otimes \beta_m \otimes 1)=0$.  Moreover, for any $t>1$,
		\[ \phi_t (f_{i_1}, \ldots, f_{i_t})(1  \otimes \beta_m \otimes \gamma_1 \otimes 1) = v_t (f_{i_1}, \ldots, f_{i_t}) H_2 (1  \otimes \beta_m \otimes \gamma_1 \otimes 1) =0,  \]
		and 
		\begin{align*}
			\phi_{n-1}(f_1, \ldots, \hat{f}_j, \dots, f_n)(1 \otimes \beta_1 \cdots \beta_{m-1} \otimes {f_j}(\beta_m \gamma_1) \otimes 1)=0
		\end{align*}
		since there is no nonzero path $w$ appearing as a summand in $f_j(\beta_m \gamma_1)$ 
		such that $\beta_{m-1} w=0$. 
		
		For $r>1$, arguing as above, 
		it is clear that 
		\[v_n (f_1, \ldots, f_n) (1 \otimes \beta_1 \cdots \beta_{m-1} \otimes \beta_m \otimes \gamma_1 \otimes 1)=0.\]
		Now, from Lemma \ref{lemma:phi=0-RBR}, we have
		$$\phi_t (f_{i_1}, \ldots, f_{i_t})(1 \otimes  \gamma_1  \cdots \gamma_{i-1} \otimes \gamma_i \otimes 1)=0, \qquad \forall t \geq 1. $$
		Also, by inductive hypothesis  
		\[ \phi_t (f_{i_1}, \ldots, f_{i_t})(1  \otimes \beta_1 \cdots \beta_m \otimes \gamma_1 \cdots \gamma_{i-1} \otimes 1) =  0, \qquad \forall {t\geq 2}.  \]
		Then, $\phi_n(f_1, \ldots, f_n)(1\otimes u \otimes v \otimes 1)$ is equal to a linear combination of terms of the form
		\begin{align*}
			\phi_{n-1}(f_1, \ldots, \hat{f}_j, \dots, f_n)(1  \otimes \beta_1 \cdots \beta_{m-1} f_j(\beta_m \gamma_1)\gamma_2 \cdots \gamma_{i-1} \otimes \gamma_i \otimes 1).
		\end{align*}
		Finally, this vanishes by Lemma \ref{lemma:phi=0-RBR}. In fact, for $i>2$ we use that $\gamma_{i-1}\gamma_i\neq 0$, and, for $i=2$ we have 
		$f_j(\beta_m \gamma_1) \gamma_2 \neq 0$ by hypothesis.
	\end{proof}

	\begin{proposition} \label{Prop. 2.13} Let $A=\kk Q/I$ be a quadratic algebra and 
		let $\omega=(\alpha_1,\alpha_2, \alpha_3)\in AP_3$ such that 
		\begin{enumerate}
			\item there is no nonzero path $u$ parallel to $\alpha_1\alpha_2$ such that $u \alpha_3=0$, and 
			\item there is no nonzero path $u$ parallel to $\alpha_2\alpha_3$ such that $\alpha_1 u=0$.
		\end{enumerate}
		Then $l_n(f_1, \ldots, f_n)(\omega)=0$ for all $n \geq 2$ and for any $f_i \in B^{1}(A)[1]$. 
	\end{proposition}
	\begin{proof}
		From equation \eqref{l_n} we have that
		\begin{align*}
			l_n(f_1,\dots, f_n)(\omega)=& (-1)^n\sum\limits_{i=1}^n \phi_{n-1}(f_1,\ldots, \hat{f}_i, \dots f_n)(1\otimes f_i(\alpha_1\alpha_2)\otimes \alpha_3\otimes 1)\\ &- (-1)^n\sum\limits_{i=1}^n \phi_{n-1}(f_1,\ldots, \hat{f}_i, \ldots f_n)(1\otimes \alpha_1\otimes f_i(\alpha_2\alpha_3)\otimes 1). \end{align*}
		From Lemma \ref{lemma:phi=0-RBR} we get that the previous expression vanishes if there is no nonzero path $u$ parallel to $\alpha_1 \alpha_2$ such that 
		$u\alpha_3=0$, and no nonzero path $u$ parallel to $\alpha_2 \alpha_3$ such that 
		$\alpha_1 u=0$.  
	\end{proof}

	\section{Gentle algebras's  case}\label{sec 4} 
	In order to prove our main results for gentle algebras in the next section, we present now a series of lemmas that will refine the necessary conditions to ensure the nilpotency of the brackets $l_n(f_1, \ldots, f_n)$ appearing in the generalized Maurer-Cartan equation. We also give some relevant counterexamples to show that the weakening of the hypotheses is not possible.
	
	We will explain in detail all the steps we have to complete in order to get all the needed results for proving our main results in Theorems \ref{teo ln=0}, \ref{theo l_4 no 0 parte2} and \ref{teo MC elements}.
	
	\begin{steps}
		\item We start with a gentle algebra $A=\kk Q/I$ and prove that $l_2(f_1, f_2)=0$, for any $f_1, f_2 \in B^1(A)[1]$; 
		\item Adding the assumption that $Q$ has no double arrows and no oriented cycles, we prove that $l_3(f_1, f_2, f_3)=0$, for any $f_1, f_2, f_3 \in B^1(A)[1]$; 
		\item If $l_n(f_1, \ldots, f_n) \neq 0$ for some $n \geq 4$ and for some $f_i \in B^1(A)[1], 1 \leq i \leq n$, we prove that $Q$ has a subquiver of the form 
		\[ \xymatrix{ 
			1   \ar[r]_{\alpha_1} \ar@{.>}@/^0.7pc/[rr]^{u}  &   2 \ar[r]_{\alpha_2}         &   3  \ar[r]_{\alpha_3}  \ar@{.>}@/^0.5pc/[r]^{w'}  & 4 &  \mbox{ or} &
			1   \ar[r]_{\alpha_1} \ar@{.>}@/^0.5pc/[r]^{v'}  &   2 \ar[r]_{\alpha_2}  \ar@{.>}@/^0.7pc/[rr]^{u}       &   3  \ar[r]_{\alpha_3}   & 4;
		}\]
		\item More precisely, we prove that the subquivers in the previous step should be as follows:
		\[ \xymatrix{
			1   \ar[r]_{\alpha_1}  &   2 \ar[r]_{\alpha_2}   \ar@{.>}@/^0.7pc/[r]^{u'}       &   3  \ar[r]_{\alpha_3}  \ar@{.>}@/^0.5pc/[r]^{w'}  & 4,
			&  &
			1   \ar[r]_{\alpha_1} \ar@{.>}@/^0.5pc/[r]^{v'}  &   2 \ar[r]_{\alpha_2}  \ar@{.>}@/^0.7pc/[r]^{u'}       &   3  \ar[r]_{\alpha_3}   & 4, \mbox{ or}
			\\
			& 5 \ar[rd]^{\beta_2} \ar[rr]^{\lambda}& & 6  \ar@{.>}@/^1pc/[rrrd]^{\epsilon_2 \cdots \epsilon_p}& &\\
			1   \ar[rr]_{\alpha_1}\ar[ru]^{\beta_1} & &  2 \ar[rr]_{\alpha_2} \ar[ru]^{\epsilon_1}  & &    3  \ar[rr]_{\alpha_3} & & 4; 
		}\]
		\item We prove that if $Q$ has no subquiver of the form
		
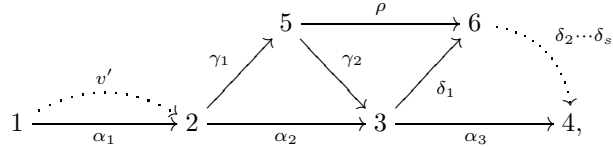
\[ \xymatrix{
			& & & 5 \ar[rd]^{\gamma_2} \ar[rr]^{\rho} & & 6 \ar@{.>}@/^1pc/[rd]^{\delta_2\cdots\delta_s} &\\
			1   \ar[rr]_{\alpha_1} \ar@{.>}@/^1pc/[rr]^{v'} & &   2 \ar[rr]_{\alpha_2} \ar[ru]^{\gamma_1}  & &    3 \ar[ru]_{\delta_1} \ar[rr]_{\alpha_3} & & 4,
		} \] then $l_n(f_1, \ldots, f_n)=0$ for all $n\geq 5$, for all $f_i \in B^1(A)[1]$;
		\item We prove that if $l_4(f_1, f_2, f_3, f_4) \neq 0$ for some $f_i \in B^1(A)[1], 1 \leq i \leq 4$,
		then $Q$ contains a subquiver of the form
		
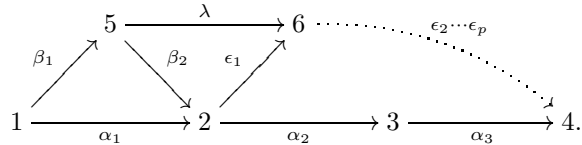
\[\xymatrix{
			& 5 \ar[rd]^{\beta_2} \ar[rr]^{\lambda}& & 6  \ar@{.>}@/^1pc/[rrrd]^{\epsilon_2\cdots\epsilon_p}& &\\
			1   \ar[rr]_{\alpha_1}\ar[ru]^{\beta_1} & &  2 \ar[rr]_{\alpha_2} \ar[ru]^{\epsilon_1}  & &    3  \ar[rr]_{\alpha_3} & & 4.
		}	\]
	\end{steps}
	
	\begin{remark} Let $A=\kk Q/I$ be a gentle algebra and 
		let $\omega=(\alpha_1,\alpha_2, \alpha_3)\in AP_3$ such that $l_n(f_1, \ldots, f_n)(\omega) \neq 0$ for some $n \geq 2$ and for some $f_i \in B^{1}(A)[1], 1 \leq i \leq n$. Then, applying Proposition \ref{Prop. 2.13}, since $A$ is gentle, we have
		\begin{enumerate}
			\item there is a nonzero path $u=v' \alpha_2$ parallel to $\alpha_1\alpha_2$, or
			\item there is a nonzero path $u=\alpha_2 w'$ parallel to $\alpha_2\alpha_3$.
		\end{enumerate}
	\end{remark}
	
	\begin{lemma} \label{l_2 =0}
		Let $A=\kk Q/I$ be a gentle algebra and 
		let $\omega=(\alpha_1,\alpha_2, \alpha_3)\in AP_3$. Then $l_2(f_1, f_2)(\omega) =0$   for all $f_1, f_2 \in B^{1}(A)[1]$.
	\end{lemma}
	\begin{proof}
		We will show that
		\begin{align*}
			l_2(f_1, f_2)(\omega) 
			= &\phi_{1}(f_1)(1\otimes f_2(\alpha_1\alpha_2)\otimes \alpha_3\otimes 1) +  \phi_{1}(f_2)(1\otimes f_1(\alpha_1\alpha_2)\otimes \alpha_3\otimes 1)\\ 
			&-  \phi_{1}(f_1)(1\otimes \alpha_1\otimes f_2(\alpha_2\alpha_3)\otimes 1) -  \phi_{1}(f_2)(1\otimes \alpha_1\otimes f_1(\alpha_2\alpha_3)\otimes 1)
		\end{align*}
		is zero. Using Lemma \ref{lemma:phi=0-RBR} we get that each summand $\phi_{1}(f_i)(1\otimes f_j(\alpha_1\alpha_2)\otimes \alpha_3\otimes 1)$
		vanishes if there is no nonzero path $v' \alpha_2$ appearing as a summand in $f_j(\alpha_1\alpha_2)$.  If it appears, and with coefficient $a_j \in \kk$,  then
		\[\phi_{1}(f_i)(1\otimes f_j(\alpha_1\alpha_2)\otimes \alpha_3\otimes 1)  = a_j \phi_{1}(f_i) (1 \otimes v' \alpha_2 \otimes \alpha_3 \otimes 1) 
		= a_j v' f_i (\alpha_2 \alpha_3) \]
		which vanishes except for the case of a nonzero path $\alpha_2 w'$ appearing as a summand in  $f_i(\alpha_2 \alpha_3)$. The same conclusion can be made for
		$\phi_{1}(f_i)(1\otimes \alpha_1\otimes f_j(\alpha_2\alpha_3)\otimes 1)$. Finally, if $v' \alpha_2$ appears as a summand in $f_j(\alpha_1\alpha_2)$ with coefficient $a_j$ and  $\alpha_2 w'$ appears as a summand in  $f_i(\alpha_2 \alpha_3)$ with coefficient $b_i$, we get that 
		\begin{align*}
			\phi_{1}(f_i)(1\otimes f_j(\alpha_1\alpha_2)\otimes \alpha_3\otimes 1) & = a_j b_i v' \alpha_2 u', \mbox{ and} \\
			\phi_{1}(f_i)(1\otimes \alpha_1\otimes f_j(\alpha_2\alpha_3)\otimes 1) & = a_j b_i v' \alpha_2 u'
		\end{align*}
		and the proof follows.
	\end{proof}

	\begin{lemma} \label{l_3 =0}
		Let $A=\kk Q/I$ be a gentle algebra and let $\omega=(\alpha_1,\alpha_2, \alpha_3)\in AP_3$. If $Q$ has no oriented cycles and no
		parallel arrows then $l_3(f_1, f_2, f_3)(\omega)=0$, for all $f_1, f_2, f_3 \in B^{1}(A)[1]$.
	\end{lemma}
	\begin{proof}
		We claim that $\phi_1(f) (1 \otimes x \otimes y \otimes 1)$ is zero if $x,y$ are both paths of length greater than $1$. Indeed, set $x=x' \nu_1, y=\nu_2 y'$, $\nu_1, \nu_2 \in Q_1$. The result is clear if $\nu_1 \nu_2 \neq 0$. In the other case,
		\[ \phi_1(f)(1 \otimes x' \nu_1 \otimes \nu_2 y' \otimes 1) =
		x' f(\nu_1 \nu_2) y'\]
		is nonzero if at least one summand of $f(\nu_1 \nu_2)$ has a nonzero path starting with $\nu_1$ and ending with $\nu_2$, which is not possible since $Q$ has no oriented cycles.
		
		Now, we have that
		\begin{align*} 
			l_3(f_1, f_2, f_3)(\omega) 
			=&-\sum_{i=1}^3  \phi_{2}(f_1, \hat{f}_i, f_3)(1\otimes f_i(\alpha_1\alpha_2)\otimes \alpha_3\otimes 1)\\ 
			& +  \sum_{i=1}^3 \phi_{2}(f_1, \hat{f}_i, f_3)(1\otimes \alpha_1\otimes f_i(\alpha_2\alpha_3)\otimes 1),
		\end{align*}
		Using Lemma \ref{lemma:phi=0-RBR} we get that the first summands vanish if $v' \alpha_2$ does not appear as a summand in $f_i (\alpha_1 \alpha_2)$,
		and the last summands vanish if $\alpha_2 w'$ does not appear as a summand in $f_i(\alpha_2 \alpha_3)$.
		If $v' \alpha_2$ appears as a summand in $f_i (\alpha_1 \alpha_2)$ with coefficient $a_i$
		then
		\begin{align*}
			\phi_{2}&(f_j, f_k)(1\otimes f_i(\alpha_1\alpha_2)\otimes \alpha_3\otimes 1)  = a_i \phi_{2}(f_j, f_k) (1 \otimes v' \alpha_2 \otimes \alpha_3 \otimes 1) \\
			& = a_i \phi_1(f_j) (1 \otimes v' \otimes f_k (\alpha_2 \alpha_3) \otimes 1) + a_i \phi_1(f_k) (1 \otimes v' \otimes f_j (\alpha_2 \alpha_3) \otimes 1).
		\end{align*} 
		As $Q$ has no parallel arrows, and $v'$ is parallel to $\alpha_1$, then $v'=v''\beta$ with $v''$ a path of positive length and $\beta\in Q_1$.
		So 
		$\phi_1(f_k) (1 \otimes v' \otimes f_j (\alpha_2 \alpha_3) \otimes 1)$ vanishes except for $f_j(\alpha_2\alpha_3)$ having an arrow $\mu$ as a summand, with $\beta \mu =0$. Hence, 
		\[\phi_1(f_k) (1 \otimes v''\beta \otimes \mu \otimes 1)= v'' f_k(\beta \mu)\]
		is zero since there is no nonzero path parallel to $\beta\mu$ starting with $\beta$. Therefore, $\phi_{2}(f_j, f_k)(1\otimes f_i(\alpha_1\alpha_2)\otimes \alpha_3\otimes 1)=0$. 
		
		If $\alpha_2 w'=\alpha_2\delta_1\cdots\delta_s$  appears as a summand in $f_i (\alpha_2 \alpha_3)$ with coefficient $b_i$,
		then
		\begin{align*}
			\phi_{2}(f_j, f_k)(1\otimes \alpha_1\otimes& f_i(\alpha_2\alpha_3)\otimes 1)  = b_i \phi_{2}(f_j, f_k) (1 \otimes  \alpha_1 \otimes \alpha_2 w' \otimes 1) \\
			= &-\sum_{r=1}^s b_i v_{2}(f_j, f_k) (1 \otimes  \alpha_1 \otimes \alpha_2 \delta_1\cdots\delta_{r-1}\otimes \delta_r \otimes 1)\delta_{r+1}\cdots \delta_s \\
			=& - b_i \phi_{1}(f_j) (1 \otimes  f_k(\alpha_1 \alpha_2)\otimes \delta_1 \otimes 1)\delta_{2}\cdots \delta_s \\
			&- b_i \phi_{1}(f_k) (1 \otimes  f_j(\alpha_1 \alpha_2)\otimes \delta_1 \otimes 1)\delta_{2}\cdots \delta_s.
		\end{align*}
		Now, $\phi_{1}(f_k) (1 \otimes  f_j(\alpha_1 \alpha_2)\otimes \delta_1 \otimes 1)\delta_{2}\cdots \delta_s$ vanishes except for $f_j(\alpha_1\alpha_2)$ having a path $\beta_1\cdots\beta_m$ as a summand, with $\beta_m \delta_1 =0$. In that case we obtain
		\[\beta_1\cdots\beta_{m-1}f_k(\beta_m\delta_1)\delta_2\cdots\delta_s\] 
		which is nonzero if $\beta_1\cdots\beta_{m-1}f_k(\beta_m\delta_1)$ has as summand a path ending with $\delta_1$, since $s>1$. This is not possible because $Q$ has no oriented cycles. Then $\phi_{2}(f_j, f_k)(1\otimes \alpha_1\otimes f_i(\alpha_2\alpha_3)\otimes 1)=0 $ and the lemma follows.
	\end{proof}
	
	If Q has oriented cycles or parallel arrows the previous result does not hold, as we can see in the following examples.

	\begin{example} \label{parallel arrows}
		\cite[Example 4.3]{RRB}.  Let $A= \kk Q/I$ with quiver
		\[ \xymatrix{1 \ar[r]_{\alpha_1}&2\ar@<1ex>[r]^{\gamma}\ar[r]_{\alpha_2}&3\ar@<1ex>[r]^{\delta}\ar[r]_{\alpha_3}&4}\]
		and $I=<\alpha_1\alpha_2, \alpha_2\alpha_3, \gamma\delta >$.  Let $f=(\alpha_1\alpha_2||\alpha_1\gamma)+(\alpha_2\alpha_3|| \alpha_2\delta)+(\gamma\delta||\gamma\alpha_3+\alpha_2\delta)$. 
		One can check that
		\[l_n (f, \ldots, f) {(\alpha_1,\alpha_2,\alpha_3)} = \begin{cases} 
			(-1)^{\frac{n-1}{2}}n!  \  ( (\alpha_1 \alpha_2, \alpha_2 \alpha_3) || \alpha_1 \gamma \alpha_3), & \mbox{if $n$ is odd;} \\
			0,  & \mbox{otherwise.} \\
		\end{cases}\]
	\end{example}
	\begin{example} \label{oriented cycles}
		Let $A= \kk Q/I$ with quiver
		
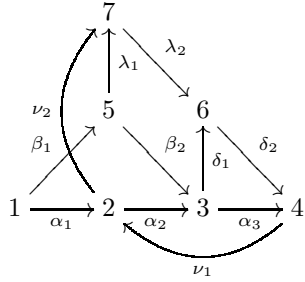
\[ \xymatrix{ 
			& 7 \ar[rd]^{\lambda_2} & & \\
			& 5 \ar[u]_{\lambda_1} \ar[rd]^
			{\beta_2}&6 \ar[rd]^{\delta_2}&\\
			1\ar[r]_{\alpha_1}\ar[ru]^{\beta_1}&2 \ar@/^1.5pc/[uu]^{\nu_2}\ar[r]_{\alpha_2}&3\ar[u]_{\delta_1}\ar[r]_{\alpha_3} &4 \ar@/^1.5pc/[ll]^{\nu_1}}
		\]
		and $I=<\alpha_1\alpha_2, \alpha_2\alpha_3, \nu_1\nu_2, \beta_2\delta_1, \beta_1\lambda_1, \delta_2\nu_1,   \nu_2\lambda_2, \lambda_2\delta_2>$.  Let
		\begin{align*}
			f=(\alpha_1\alpha_2||\beta_1\beta_2+ & \beta_1\beta_2\alpha_3\nu_1\alpha_2)+(\alpha_2\alpha_3||\alpha_2\delta_1\delta_2) \\
			+&(\beta_2 \delta_1||\beta_2\alpha_3\nu_1\alpha_2\delta_1+\lambda_1\lambda_2)
			+(\beta_1\lambda_1||\alpha_1\nu_2)+(\nu_2\lambda_2||\alpha_2\delta_1). 
		\end{align*}
		One can check that
		\[l_n (f, \ldots, f) {(\alpha_1,\alpha_2,\alpha_3)} = \begin{cases} 
			-n! \beta_1\beta_2\alpha_3\nu_1\alpha_2\delta_1\delta_2, & \mbox{if $3$ divides $n$;} \\
			0,  & \mbox{otherwise.} \\
		\end{cases}\]
	\end{example}

	Now we present a result concerning the nonvanishing of $l_n$ for $n \geq 4$.
	\begin{proposition}\label{prop 1}
		Let $A=\kk Q/I$ be a gentle algebra and 
		let $\omega=(\alpha_1,\alpha_2, \alpha_3)\in AP_3$ such that $l_n(f_1, \ldots, f_n)(\omega)\neq 0$ for some $n \geq 4$ and for some $f_i \in B^{1}(A)[1], 1 \leq i \leq n$. Then 
		\begin{enumerate}
			\item
			there is a nonzero path $u$ parallel to $\alpha_1\alpha_2$ and there is a nonzero path $w=\alpha_2 w'$ parallel to  $\alpha_2\alpha_3$ such that $uw'=0$ 
			\[ \xymatrix{ 
				1   \ar[r]_{\alpha_1} \ar@{.>}@/^0.7pc/[rr]^{u}  &   2 \ar[r]_{\alpha_2}         &   3  \ar[r]_{\alpha_3}  \ar@{.>}@/^0.5pc/[r]^{w'}  & 4
			} , \mbox{ or} \]
			\item
			there is a nonzero path $v= v' \alpha_2$ parallel to $\alpha_1\alpha_2$ and there is a nonzero path $u$ parallel to  $\alpha_2\alpha_3$ such that $v' u =0$ 
			\[ \xymatrix{ 
				1   \ar[r]_{\alpha_1} \ar@{.>}@/^0.5pc/[r]^{v'}  &   2 \ar[r]_{\alpha_2}  \ar@{.>}@/^0.7pc/[rr]^{u}       &   3  \ar[r]_{\alpha_3}   & 4.
			}\]
		\end{enumerate}
	\end{proposition}
	\begin{proof}
		From  equation \eqref{l_n} we have that $l_n(f_1, \dots, f_n)(\omega)\neq 0$ implies that, for some $i$, 
		\begin{align} \label{caso a}
			\phi_{n-1}(f_1,\ldots, \hat{f}_i, \ldots, f_n)(1\otimes \alpha_1\otimes f_i(\alpha_2\alpha_3)\otimes 1) \neq 0, \mbox{ or } \\ \label{caso b}
			\phi_{n-1}(f_1,\ldots, \hat{f}_i, \ldots , f_n)(1\otimes f_i(\alpha_1\alpha_2)\otimes \alpha_3\otimes 1) \neq 0.
		\end{align}
		If \eqref{caso a} holds, we can affirm that a nonzero path $w= \alpha_2w'$ appears as a summand in $f_i(\alpha_2\alpha_3)$. Applying Lemma \ref{lemma: no parallel} to 
		\begin{align*}
			\phi_{n-1}&(f_1,\ldots, \hat{f}_i, \ldots, f_n)(1\otimes  \alpha_1 \otimes \alpha_2w'\otimes 1)
		\end{align*}
		we get
		that there exists a nonzero path $u$ parallel to $\alpha_1\alpha_2$ such that $uw'=0$. 
		We proceed analogously for \eqref{caso b}. 
	\end{proof} 
	
	In the next example we show that the previous results can not be extended to the case
	of string algebras.
	
	\begin{example}
		Let $A=kQ/I$ be the string algebra given by the quiver
		
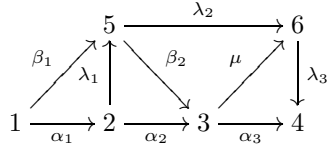
\[\xymatrix{ 
			& 5\ar[rd]^{\beta_2}\ar[rr]^{\lambda_2}& & 6\ar[d]^{\lambda_3}&\\
			1\ar[ru]^{\beta_1}\ar[r]_{\alpha_1}& 2\ar[u]^{\lambda_1}\ar[r]_{\alpha_2}& 3\ar[r]_{\alpha_3}\ar[ru]^{\mu}&4 }
		\]
		with $I=<\alpha_1\alpha_2, \alpha_2\alpha_3, \lambda_1\beta_2, \beta_1\lambda_2, \beta_2\alpha_3, \alpha_2\mu, \mu\lambda_3>$. Set $$f=(\alpha_1\alpha_2||\beta_1\beta_2)+(\beta_2\alpha_3||\lambda_2\lambda_3)+(\beta_1\lambda_2||\beta_1\beta_2\mu+\alpha_1\lambda_1\lambda_2)+(\mu\lambda_3||\alpha_3).$$  Then, one can check that
		\[l_{3k}(f, \ldots, f)(\alpha_1,\alpha_2,\alpha_3)=c_k \, \alpha_1\lambda_1\lambda_2\lambda_3\neq 0\] 
		for some nonzero $c_k \in \kk$.
	\end{example}
	
	We have seen in Examples \ref{parallel arrows} and \ref{oriented cycles} that gentle algebras with non-nilpotent $L_\infty$-structure can be found
	when the quiver has parallel arrows or oriented cycles.
	
	Now, assuming the hypothesis about no oriented cycles and no parallel arrows, we will find conditions for nilpotency.

	\begin{remark}\label{loop argument}
		In the subsequent proofs, we will use the following argument to show that $$\phi_n(f_1,\ldots, f_n)(1\otimes x_0 \otimes y_0\otimes 1)=0$$
		for any $n\geq 3$. The recursive definition of $\phi_n$ implies that
		$\phi_n(f_1,\ldots, f_n)(1\otimes x_0 \otimes y_0\otimes 1)$ will be expressed in terms of $\phi_r(f_{i_1}\ldots, f_{i_r}) (1\otimes x_t \otimes y_t \otimes 1)$ for $1 \leq r <n$ and for several possible arguments $1\otimes x_t \otimes y_t \otimes 1$. The desired vanishing will follow by the equalities $$\phi_1(f_j)(1\otimes x_t \otimes y_t \otimes 1)=0
		\qquad \mbox{ and } \qquad \phi_2(f_{i_1},f_{i_2})(1\otimes x_t\otimes y_t\otimes 1)=0$$
		for all possible arguments. We will explain in detail this procedure in the following example.
		
	\end{remark}
	\begin{example}\label{ejemplo s=1}
		Let $Q$ be the quiver
		
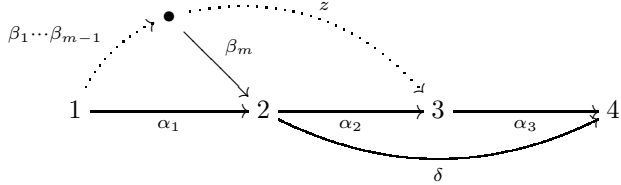
\[ \xymatrix{
			& \bullet\ar[dr]^{\beta_m}\ar@{.>}@/^1.5pc/[drrr]^z &&\\
			1    \ar@{.>}@/^.5pc/[ru]^{\beta_1 \cdots \beta_{m-1}} \ar[rr]_{\alpha_1} &  &  2 \ar[rr]_{\alpha_2}  \ar@/_1.5pc/[rrrr]_{\delta}       &  & 3  \ar[rr]_{\alpha_3}  & & 4
		}\]
		with $\beta_1 \cdots \beta_m \alpha_2$ and $z \alpha_3$ nonzero paths, $\delta$ an arrow and $(\alpha_1,\alpha_2, \alpha_3)\in AP_3$.
		For any $n\geq 4$, we will see that $\phi_{n}(f_1, \ldots, f_n) (1\otimes \beta_1\cdots\beta_m\alpha_2\otimes \alpha_3\otimes 1)=0$.
		We have
		\begin{align}
			\label{argumento1}  \phi_{n}(f_1, \ldots, f_n) &(1\otimes \beta_1\cdots\beta_m\alpha_2\otimes \alpha_3\otimes 1)
			\\ \notag
			&=v_{n}(f_1,\ldots, f_n) (1\otimes \beta_1\cdots\beta_m\otimes \alpha_2\otimes \alpha_3\otimes 1) \\ 
			\notag
			&=(-1)^{n} \sum_{j=1}^n  \phi_{n-1}(f_1,\ldots,\hat{f}_j, \dots, f_n) (1\otimes \beta_1\cdots\beta_m\otimes f_j(\alpha_2 \alpha_3)\otimes 1).
		\end{align}
		Since the unique nonzero path parallel to $\alpha_2\alpha_3$ is $\delta$, the only argument that can appear for $\phi_{n-1}(f_1,\dots,\hat{f}_j, \dots, f_n)$ is 
		\begin{equation*}\label{argumento2}
			1\otimes \beta_1\cdots\beta_m\otimes \delta \otimes 1.
		\end{equation*} 
		Now we proceed,
		\begin{align*} 
			\phi_{n-1}&(f_1,\ldots,\hat{f}_j, \ldots, f_n) (1\otimes \beta_1\cdots \beta_m\otimes \delta\otimes 1)\\
			\notag  &= v_{n-1}(f_1,\ldots,\hat{f}_j, \ldots, f_n) (1\otimes \beta_1\cdots\beta_{m-1}\otimes \beta_m\otimes \delta\otimes 1)\\ 
			&= (-1)^{n-1} \sum_{\substack{k=1 \\k \neq j}}^n \phi_{n-2}(f_1,\ldots,\hat{f}_j,\ldots,\hat{f}_k, \dots, f_n) (1\otimes \beta_1\cdots\beta_{m-1}\otimes f_k(\beta_m \delta) \otimes 1)
		\end{align*}
		and the unique nonzero path parallel to $\beta_m\delta$ is $z \alpha_3$, so the only argument that can appear for $\phi_{n-2}(f_1,\ldots,\hat{f}_j,\ldots,\hat{f}_k, \ldots, f_n)$ is 
		\begin{equation*}\label{argumento3}
			1\otimes \beta_1\cdots\beta_{m-1}\otimes z\alpha_3 \otimes 1.
		\end{equation*}
		Since $Q$ has no oriented cycles, there is no nonzero path from $1$ to $4$ ending with $\alpha_3$.
		Hence, 
		\begin{align*}
			\phi_{n-2}&(f_1,\ldots,\hat{f}_j,\ldots,\hat{f}_k, \ldots, f_n) (1\otimes \beta_1\cdots\beta_{m-1}\otimes z\alpha_3 \otimes 1)\\
			&= v_{n-2}(f_1,\ldots,\hat{f}_j,\ldots,\hat{f}_k, \dots, f_n) (1\otimes \beta_1\cdots\beta_{m-1}\otimes z\otimes \alpha_3 \otimes 1).
		\end{align*}
		By the definition of $v_{n-2}$ this is a linear combination of elements of the form
		\begin{align*}
			\phi_{n-2-t}(f_{i_1},\ldots, f_{i_{n-2-t}})(1\otimes \phi_t(f_{i_{n-1-t}},\ldots, f_{i_{n-2}})(1\otimes \beta_1\cdots\beta_{m-1}\otimes z\otimes 1)\otimes \alpha_3\otimes 1).
		\end{align*}
		The only nonzero path parallel to $\beta_1\cdots\beta_{m-1}z$ is $\beta_1\cdots\beta_m\alpha_2$. Then, the unique nonzero elements that may appear are of the form
		\begin{align*} \phi_{n-2-t}(f_{i_1},\ldots, f_{i_{n-2-t}})(1\otimes \beta_1\cdots\beta_m\alpha_2 \otimes \alpha_3\otimes 1).
		\end{align*}
		As the argument of the last equation coincides with the argument of \eqref{argumento1}, we can affirm that the only possible arguments appearing in $\phi_l (f_{j_1},\ldots, f_{j_l})$ are 
		$$1\otimes \beta_1\cdots\beta_m\alpha_2\otimes \alpha_3\otimes 1, \,  1\otimes \beta_1\cdots\beta_m\otimes \delta\otimes 1,\,  \mbox{ and }  \, 1\otimes \beta_1\cdots\beta_{m-1}\otimes z \alpha_3\otimes 1.$$
		One can check that $\phi_1(f_i)$ and $\phi_2(f_i,f_j)$ vanish  in all of them, therefore $$\phi_{n}(f_1, \ldots, f_n) (1\otimes \beta_1\cdots\beta_m\alpha_2\otimes \alpha_3\otimes 1)=0.$$
	\end{example}

	\
	
	From Proposition \ref{prop 1} we need to study in detail those algebras whose quivers contain subquivers of one of the following two cases:
	\begin{itemize}
		\item [Case A:]
		$(\alpha_1,\alpha_2, \alpha_3)\in AP_3$,  nonzero paths $u, w'$ such that $uw' =0$ 
		\[ \xymatrix{ 
			1   \ar[r]_{\alpha_1} \ar@{.>}@/^0.7pc/[rr]^{u}  &   2 \ar[r]_{\alpha_2}         &   3  \ar[r]_{\alpha_3}  \ar@{.>}@/^0.5pc/[r]^{w'}  & 4
		}, \mbox{ or}\]
		\item [Case B:]
		$(\alpha_1,\alpha_2, \alpha_3)\in AP_3$,
		nonzero paths $v',w$ such that $v'w =0$ 
		\[ \xymatrix{ 
			1   \ar[r]_{\alpha_1} \ar@{.>}@/^0.5pc/[r]^{v'}  &   2 \ar[r]_{\alpha_2}  \ar@{.>}@/^0.7pc/[rr]^{u}       &   3  \ar[r]_{\alpha_3}   & 4.
		} \]
	\end{itemize}
	
	\subsection{Case A}
	In this case we get the following result.
	
	\begin{proposition} \label{proposition A}
		Let $A=\kk Q/I$ be a gentle algebra. Assume $Q$ has no oriented cycles, no parallel arrows and it contains a subquiver of the form  \[ \xymatrix{ 
			1   \ar[r]_{\alpha_1} \ar@{.>}@/^0.7pc/[rr]^{u}  &   2 \ar[r]_{\alpha_2}         &   3  \ar[r]_{\alpha_3}  \ar@{.>}@/^0.5pc/[r]^{w'}  & 4
		}  \]
		with  $\omega=(\alpha_1,\alpha_2, \alpha_3)\in AP_3$, $u,w'$ nonzero paths and $uw'=0$. If $l_n(f_1, \dots, f_n)(\omega)\neq 0$ for some  $n \geq 4$ and for some $f_i \in B^{1}(A)[1], 1 \leq i \leq n$, then  $u= \alpha_1 u'$. 
	\end{proposition}
	\begin{proof}
		Write $u=  \beta_1 \cdots \beta_m$, $\beta_1 \neq   \alpha_1$ and  $w'= \delta_1 \cdots \delta_s$ with $\beta_m \delta_1=0$. 
		Since $Q$ has no oriented cycles, the unique nonzero path from $1$ to $4$ is 
		$u \alpha_3$.
		From \eqref{l_n} and Lemma \ref{lemma:phi=0-RBR}, we get that $l_n(f_1, \dots, f_n)(\omega)\neq 0$ implies that 
		\begin{equation}
			\label{phi n-1 Ai}
			\phi_{n-1}(f_1,\dots, \hat{f}_i,  \dots,  f_n)(1\otimes \alpha_1 \otimes \alpha_2w' \otimes 1)
		\end{equation}
		is not zero for some $i$. 
		Observe that \eqref{phi n-1 Ai} equals
		\[ - v_{n-1}(f_1,\dots, \hat{f}_i,  \dots,  f_n)(1 \otimes \alpha_1 \otimes \alpha_2 \delta_1  \cdots \delta_{s-1} \otimes \delta_s \otimes 1)\]
		since all the other terms do not end with $\alpha_3$. 
		By definition, this is a linear combination of terms of the form
		\begin{align} \label{phi_n-1-t A}
			\phi_{n-1-t}(f_{i_1},\dots, f_{i_{n-1-t}})(1\otimes\phi_t (f_{i_{n-t}},\dots, f_{i_{n-1}})(1\otimes\alpha_1\otimes \alpha_2  \delta_1\cdots \delta_{s-1}\otimes 1)\otimes \delta_s\otimes 1).
		\end{align}
		Since $Q$ has no oriented cycles, from Lemma \ref{lemma:phi=0-RBR} we get that the unique possibility for the nonvanishing of $\phi_{n-1-t}$ is that there exists a nonzero path $\alpha_1\mu_1\cdots\mu_r$ parallel to $\alpha_1\alpha_2 \delta_1\cdots \delta_{s-1}$
		with $\mu_r\neq\delta_{s-1}$.
		
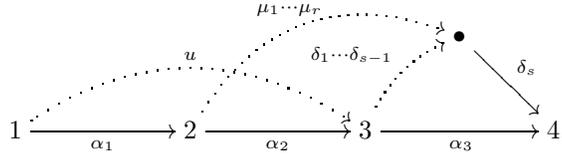
\[
		\xymatrix{
			& & & & & \bullet  \ar[rd]^{\delta_s} \\
			1 \ar[rr]_{\alpha_1} \ar@{.>}@/^2pc/[rrrr]^u & & 2 \ar[rr]_{\alpha_2} \ar@{.>}@/^2pc/[rrru]^{\mu_1 \cdots \mu_r} & & 3 \ar[rr]_{\alpha_3} \ar@{.>}@/^.5pc/[ru]^{\delta_1 \cdots \delta_{s-1}}  & & 4.
		}
		\]
		In this case, as the only nonzero path from $1$ to $4$ starts with $\beta_1$, \eqref{phi_n-1-t A} vanishes for $n-1-t=1$, and, for $n-1-t \geq 2$, it can only be of the form
		\begin{align} \notag
			\phi_{n-1-t}&{(f_{j_{1}},\dots, f_{j_{n-t-1}})}(1\otimes \alpha_1\mu_1\cdots\mu_r\otimes\delta_s\otimes 1)
			\\
			&=v_{n-1-t}{(f_{j_{1}},\dots, f_{j_{n-t-1}})}(1\otimes \alpha_1\mu_1\cdots\mu_{r-1} \otimes\mu_r\otimes\delta_s\otimes 1)\label{case A n-1-t} .
		\end{align}
		For $r>1$, \eqref{case A n-1-t} is zero because there are no nonzero paths parallel to $\mu_r\delta_s$ since  $Q$ has no oriented cycles.  
		Finally, for  $r=1$, the only nonzero path parallel to $\mu_1\delta_s$ is  $\alpha_2w'$.  Therefore,  
		\eqref{case A n-1-t} 
		can only be a linear combination of terms of the form
		\begin{align*}
			\phi_{n-2-t}(f_{k_1},\dots, f_{k_{n-2-t}})(1\otimes\alpha_1\otimes \alpha_2 w'\otimes 1)
		\end{align*}
		whose arguments coincide with the one of \eqref{phi n-1 Ai}. Hence, the only arguments appearing in $\phi_p (f_{t_1},\dots, f_{t_p})$ are $1\otimes \alpha_1\otimes \alpha_2 w' \otimes 1 $ and $1\otimes \alpha_1\mu_1\otimes \delta_s \otimes 1 $. Applying Remark \ref{loop argument} we get that $l_n(f_1,\dots,f_n)({\omega})$ vanishes.
	\end{proof}
	
	\subsection{Case B}
	Assume that each quiver in this subsection has no oriented cycles, no parallel arrows, and 
	contains a subquiver
	\[ \xymatrix{ 
		1   \ar[r]_{\alpha_1} \ar@{.>}@/^0.5pc/[r]^{v'}  &   2 \ar[r]_{\alpha_2}  \ar@{.>}@/^0.7pc/[rr]^{u}       &   3  \ar[r]_{\alpha_3}   & 4
	} \]
	with $\omega =(\alpha_1,\alpha_2, \alpha_3)\in AP_3$, $v'= \beta_1 \dots \beta_m$ and $u=\delta_1 \dots \delta_s$ nonzero paths, and $\beta_m \delta_1 =0$. Observe that, since $Q$ has no double arrows, $m >1$. 
	We start this subsection with a series of technical lemmas concerning the vanishing of $v_n$.
	
	\begin{lemma}\label{lema A}
		With the notation above, if $\delta_s\neq \alpha_3$, then 
		\begin{equation*}
			v_n(f_1,\dots,f_n)(1\otimes \beta_1\cdots\beta_{m-1}\otimes \beta_m\otimes \delta_1\otimes 1)\delta_2\cdots\delta_s =0
		\end{equation*}
		for any $n \geq 3$ or $m \geq 3$. Moreover, for $n=m=2$, if $s>1$ and $\lambda$ is an arrow in $Q$ parallel to  $\beta_2\delta_1$, then there exist $f_1,f_2\in B^{1}(A)[1]$ such that 
		$$v_2(f_1, f_2)(1\otimes \beta_1\otimes \beta_2\otimes \delta_1\otimes 1)\delta_2\cdots\delta_s \neq 0.$$
	\end{lemma}
	\begin{proof}
		{From Remark \ref{rmk 4.6} we have that 
			\begin{align}
				v_n (f_1, \dots, &f_n)(1 \otimes \beta_1\cdots\beta_{m-1} \otimes \beta_m  \otimes  \delta_1\otimes 1) \delta_2\cdots\delta_s = \notag \\  \label{vn A}
				& (-1)^{n+1} \sum \phi_{n-1}(f_1,\dots,\hat{f}_i,\dots,f_n)(1\otimes \beta_1\cdots\beta_{m-1}\otimes f_i(\beta_m\delta_1)\otimes 1) \delta_2\cdots\delta_s .
		\end{align}}
		Since $\delta_s \neq \alpha_3$, the unique nonzero path from $1$ to $4$ is 
		$\alpha_1 u$. 
		The case $s=1$ was considered in some of the steps developed in Example \ref{ejemplo s=1}. 
		
		Assume $s>1$. If $n\geq 3$,  equation \eqref{vn A} could be nonzero only if there exists a path $\lambda_1\dots\lambda_h$  parallel to $\beta_m\delta_1$. 
		
		Consider $h=1$. For $m=2$, \eqref{vn A} vanishes.  For $m \geq 3$, arguing as in Example \ref{ejemplo s=1}
		one can check  that the  only arguments appearing in $\phi_t (f_{j_1},\dots, f_{j_t})$ are $1\otimes \beta_1\dots \beta_{m-1}\otimes \lambda \otimes 1 $ and $1\otimes \beta_1\dots \beta_{m-2}\otimes z \lambda \otimes 1 $, where $z$ is a nonzero path parallel to $\beta_{m-1}$. Applying Remark \ref{loop argument} we get that \eqref{vn A} vanishes.

		Assume $h>1$.  In this case, \eqref{vn A} is a linear combination of terms of the form 
		\begin{align}\label{h>1}
			\phi_{n-1-t}(f_{j_1},\dots, f_{j_{n-1-t}})( 1 \otimes x_t \otimes \lambda_h \otimes 1)\delta_2\cdots\delta_s, \mbox{ with }\\
			\notag
			x_t= \phi_t(f_{j_{n-t}},\dots, f_{j_{n-1}}) (1\otimes \beta_1 \cdots \beta_{m-1} \otimes \lambda_1\cdots \lambda_{h-1}\otimes 1).
		\end{align}
		If $\lambda_1=\beta_m$, they vanish. Assume $\lambda_1\neq \beta_m$.
		Since $Q$ has no oriented cycles, the only nonzero path parallel to $\beta_1 \cdots \beta_{m-1} \lambda_1\cdots \lambda_{h-1}$ could be $\beta_1\cdots\beta_m\alpha_2 \nu_1\cdots\nu_p$: 
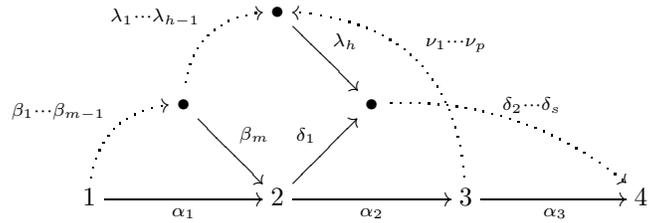
		\[\xymatrix{
			& & \bullet  \ar[rd]^{\lambda_h}\\
			& \bullet \ar[rd]^{\beta_m} \ar@{.>}@/^1pc/[ru]^{\lambda_1 \cdots \lambda_{h-1}} & & \bullet  \ar@{.>}@/^1pc/[rrrd]^{\delta_2 \cdots \delta_s}\\
			1 \ar[rr]_{\alpha_1} \ar@{.>}@/^1pc/[ru]^{\beta_1 \cdots \beta_{m-1}} & & 2 \ar[rr]_{\alpha_2} \ar[ru]^{\delta_1}& & 3  \ar@{.>}@/_2pc/[lluu]_{\nu_1\cdots\nu_p}\ar[rr]_{\alpha_3}& & 4.
		}
		\]
		If $\nu_p=\lambda_{h-1}$ or $n-1-t=1$, then \eqref{h>1} vanishes. 
		If not, then \eqref{h>1} can only have terms of the form 
		\[ v_d(f_{j_1},\dots, f_{j_d})(1\otimes\beta_1\cdots\beta_m\alpha_2\nu_1\cdots\nu_{p-1}\otimes \nu_p\otimes\lambda_h\otimes 1)\delta_2\cdots\delta_s\]
		with $d\geq 2$, which is equal to a sum of terms of the form
		$$(-1)^{d+1}\phi_{d-1}(f_{k_1},\dots, f_{k_{d-1}})(1\otimes\beta_1\cdots\beta_m\alpha_2\nu_1\cdots\nu_{p-1}\otimes f_k(\nu_p\lambda_h)\otimes 1)\delta_2\cdots\delta_s.$$
		Now, whether $f_k(\nu_p\lambda_h)$ ends with $\lambda_h$ or with $\delta_1$, we have cycles, which is a contradiction. Hence \eqref{vn A} vanishes for $n \geq 3$.
		
		If $n=2$ and $m>2$, then \eqref{vn A} is zero since there is no nonzero path from $1$ to $4$ starting with $\beta_{1}$.
		Finally, if $n=2$ and $m=2$, then \eqref{vn A}  
		is zero unless there exists a path $\lambda_1\dots \lambda_h$ appearing as a summand in $f_i(\beta_2\delta_1)$. It also vanishes if $h >1$ since
		$$\phi_1(f_{j})(1\otimes \beta_1\otimes \lambda_1\dots \lambda_h\otimes 1)=f_j(\beta_1\lambda_1)\lambda_2\cdots\lambda_h$$
		and $\lambda_h\delta_2=0$. 
		If $h=1, f_1=(\beta_1 \lambda_1||\alpha_1 \delta_1)$ and $ f_2=(\beta_2 \delta_1||\lambda_1)$, then  
		\[v_2(f_1, f_2)(1\otimes \beta_1\otimes \beta_2\otimes \delta_1
		\otimes 1)\delta_2\cdots\delta_s = \alpha_1 u \neq 0\]
		and the existence of $\lambda, \delta_s$ and $\alpha_3$ implies $s>1$.
		We have proved that, except for $n=m=2$, \eqref{vn A} vanishes and the lemma follows.
	\end{proof}

	\begin{lemma}\label{lema s=2}
		With the notation above,  if $s=2$, then
		\begin{equation*}
			v_n(f_1,\dots,f_n)(1\otimes \beta_1\cdots\beta_{m}\otimes \delta_1 \otimes \delta_2\otimes 1)=0
		\end{equation*}
		for any $n\geq2$.
	\end{lemma}
	
	\begin{proof}
		We proceed by induction on $n$. 
		Since $Q$ has no double arrows and no oriented cycles, $\delta_2 \neq \alpha_3$ and $\alpha_1 \delta_1 \delta_2$ is the unique nonzero path from $1$ to $4$. By definition 
		\[v_n(f_1,\ldots,f_n)(1\otimes \beta_1\cdots\beta_{m}\otimes \delta_1 \otimes \delta_2\otimes 1)\] is a linear combination of terms of the form
		\begin{align}\label{vn s2}
			\phi_{n-t}(f_{i_1},\ldots,f_{i_{n-t}})(1\otimes \phi_t(f_{i_{n-t+1}}, \ldots , f_{i_n})(1 \otimes  \beta_1\cdots\beta_{m}\otimes \delta_1 \otimes 1) \otimes \delta_2\otimes 1).
		\end{align} 
		For $n=2$, we have that \eqref{vn s2}
		starts with $\beta_1$, thus it vanishes.
		For $n\geq 3$, the possible nonzero paths parallel to $\beta_1\cdots\beta_m\delta_1$ are $\alpha_1\delta_1$ and
		$\beta_1\cdots\beta_m\alpha_2\mu_1\cdots\mu_h$, where $\mu_1\cdots\mu_h$ is a path from $3$ to $s(\delta_2)$:
		
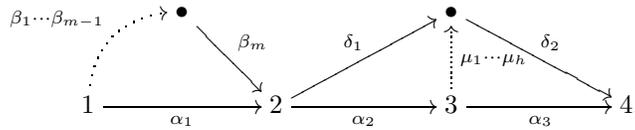
\[\xymatrix{
			& \bullet \ar[rd]^{\beta_m}  & & & \bullet  \ar[rrd]^{\delta_2}\\
			1 \ar[rr]_{\alpha_1} \ar@{.>}@/^1pc/[ru]^{\beta_1 \cdots \beta_{m-1}} & & 2 \ar[rr]_{\alpha_2} \ar[rru]^{\delta_1}& & 3  \ar@{.>}[u]_{\mu_1 \cdots \mu_h}\ar[rr]_{\alpha_3}& & 4.
		}
		\]
		Since $\delta_1\delta_2\neq 0$, then \eqref{vn s2} can only be a linear combination of terms of the form
		\begin{equation*}
			\phi_{n-t}(f_{j_1}, \dots,f_{j_{n-t}})(1\otimes \beta_1\cdots\beta_m\alpha_2\mu_1\cdots\mu_h \otimes \delta_2\otimes 1).
		\end{equation*}
		This equals zero for $n-t=1$ and, for $n-t >1$, it equals
		\begin{equation*}
			v_{n-t}(f_{j_1}, \dots,f_{j_{n-t}})(1\otimes \beta_1\cdots\beta_m\alpha_2\mu_1\cdots\mu_{h-1}\otimes \mu_h \otimes \delta_2\otimes 1).
		\end{equation*}
		Now, concerning the existence of nonzero paths parallel to $\mu_h\delta_2$, the only possibility is $\alpha_3$ when $h=1$, since $Q$ has no oriented cycles.
		For $n-t=2$, we get a path starting in $\beta_1$, hence it vanishes. For $n-t>2$, only the terms of the form
		\begin{align*}
			\phi_{n-t-1}&(f_{k_1}, \dots,f_{k_{n-t-1}})(1\otimes \beta_1\cdots\beta_m\alpha_2\otimes \alpha_3\otimes 1)\\
			&= v_{n-t-1}(f_{k_1}, \dots,f_{k_{n-t-1}})(1\otimes \beta_1\cdots\beta_m \otimes \alpha_2\otimes \alpha_3\otimes 1) 
		\end{align*}
		may remain. Since $\phi_{p}(f_{l_1}, \ldots, f_{l_{p}}) (1 \otimes  \alpha_2 \otimes \alpha_3 \otimes 1)=0$ for any $p>1$, and $\delta_1 \delta_2$ is the only nonzero path parallel to $\alpha_2 \alpha_3$,
		the previous expression can only be a linear combination of terms of the form 
		\[ \phi_{n-t-2}(f_{l_1}, \ldots, f_{l_{n-t-2}}) (1 \otimes  \beta_1\cdots\beta_m \otimes \delta_1 \delta_2 \otimes 1). \]
		This is zero if $n-t-2=1$, and it is equal to
		\begin{align*}
			&v_{n-t-2} (f_{l_1}, \dots,f_{l_{n-t-2}})(1\otimes \beta_1\cdots\beta_{m-1}\otimes \beta_m\otimes \delta_1\otimes 1)\delta_2 \\ 
			& - v_{n-t-2} (f_{l_1}, \dots,f_{l_{n-t-2}})(1\otimes \beta_1\cdots\beta_m \otimes \delta_1\otimes \delta_2\otimes 1)
		\end{align*}
		if $n-t-2>1$.
		The first term vanishes because the only nonzero path parallel to $\beta_m\delta_1$ starts with $\beta_m$, since $t(\beta_m)=t(\alpha_1)$ and $t(\delta_1) = t(\mu_1)$. The second term vanishes by inductive hypothesis. Hence the result follows.
	\end{proof}

	\begin{lemma}\label{lema B}
		With the notation above, if $n\geq2$ and $s>2$, then
		\begin{equation*}
			v_n(f_1,\dots,f_n)(1\otimes v'\otimes \delta_1\cdots \delta_{i-1} \otimes \delta_i\otimes 1)\delta_{i+1}\cdots\delta_s =0
		\end{equation*}
		for every $i$. 
	\end{lemma}
	\begin{proof}
		We proceed by induction on $i$. Let $i=2$.
		By definition, 
		\[v_n(f_1,\dots,f_n)(1\otimes v'\otimes \delta_1\otimes \delta_2\otimes 1)\delta_{3}\cdots\delta_s\]
		is a linear combination of terms of the form
		\begin{align}\label{eq term BB}
			\phi_{n-t}(f_{j_1},\dots, f_{j_{n-t}})(1\otimes \phi_t(f_{j_{n+1-t}},\dots, f_{j_{n}})(1\otimes v'\otimes\delta_1\otimes 1)\otimes \delta_2\otimes 1)\delta_{3}\cdots\delta_s.
		\end{align}
		The only possible nonzero paths parallel to $v'\delta_1$ are $\alpha_1\delta_1$ and $v'\alpha_2\mu_1\cdots\mu_h$, where $\mu_1\cdots\mu_h$ is a nonzero path from $3$ to $s(\delta_2)$. 
		Since $\delta_1 \delta_2 \neq 0$,  \eqref{eq term BB} can only be a linear combination of terms of the form 
		\[ \phi_{n-t}(f_{j_1},\dots, f_{j_{n-t}})(1\otimes v'\alpha_2\mu_1\cdots\mu_h \otimes \delta_2\otimes 1)\delta_{3}\cdots\delta_s. \]
		For $n-t=1$,  this term vanishes since there is no nonzero path from $1$ to $4$ starting on $\beta_1$. For $n-t>1$, it is equal to
		\[ v_{n-t}(f_{j_1},\dots, f_{j_{n-t}})(1\otimes v'\alpha_2\mu_1\cdots\mu_{h-1}\otimes \mu_h \otimes \delta_2\otimes 1)\delta_{3}\cdots\delta_s.\]
		Now,  one can check that for $h=1$, this last equation equals zero since the only nonzero path parallel to $\mu_1\delta_2$ starts with $\mu_1$. 
		For $h>1$, we can only have terms of the form
		\[\phi_{n-1-t}(f_{k_1},\dots, f_{k_{n-1-t}})(1\otimes v'\alpha_2\mu_1\cdots\mu_{h-1} \otimes \nu_1\cdots\nu_p \otimes 1)\delta_{3}\dots\delta_s\]
		where $\nu_1\cdots\nu_p$ is a nonzero path from $s(\mu_h)$ to $t(\delta_2)$:
		
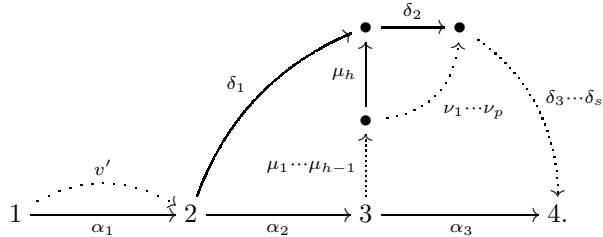
\[ \xymatrix{ & & &  &\bullet \ar[r]^{\delta_2} & \bullet \ar@{.>}@/^1pc/[rdd]^{\delta_3 \cdots \delta_s} & \\
			& & & & \bullet  \ar[u]^{\mu_h} \ar@{.>}@/_1pc/[ru]_{\nu_1\cdots\nu_p}& & &\\
			1   \ar@{.>}@/^1pc/[rr]^{v'} \ar[rr]_{\alpha_1} & &  2 \ar[rr]_{\alpha_2} \ar@/^1pc/[rruu]^{\delta_1}   &  &  3   \ar@{.>}[u]^{\mu_1 \cdots \mu_{h-1}} \ar[rr]_{\alpha_3} & & 4.}
		\]
		Arguing as in Example \ref{ejemplo s=1}
		one can check  that, for $p=1$, the  only possible arguments appearing in $\phi_c (f_{l_1},\dots, f_{l_c})$ are $1\otimes v'\alpha_2\mu_1\cdots\mu_{h-1}\otimes \nu_1 \otimes 1 $ and $1\otimes v'\alpha_2\mu_1\cdots\mu_{h-2}\otimes y \nu_1 \otimes 1$,  where $y$ is a nonzero path parallel to $\mu_{h-1}$, and for $p>1$, the  only possible arguments are $1\otimes v'\alpha_2\mu_1\cdots\mu_{h-1}\otimes \nu_1 \ldots \nu_p \otimes 1 $,  $1\otimes v'\alpha_2\mu_1\cdots\mu_{h}z\otimes \nu_p \otimes 1$, and $1\otimes v'\alpha_2\mu_1\cdots\mu_{h}\otimes \delta_2 \otimes 1 $, where $z$ is a nonzero path from $t(\mu_h)$ to $s(\nu_p)$. Applying Remark \ref{loop argument} we get that \eqref{eq term BB} vanishes.
		
		Let $i>2$. By definition, 
		\[v_n(f_1,\dots,f_n)(1\otimes v'\otimes \delta_1\cdots \delta_{i-1} \otimes \delta_i\otimes 1)\delta_{i+1}\cdots\delta_s\]
		is a linear combination of terms of the form
		\begin{align}\label{eq term BBB}
			\phi_{n-t}(f_{j_1},\dots, f_{j_{n-t}})(1\otimes \phi_t(f_{j_{n+1-t}},\dots, f_{j_{n}})(1\otimes v'\otimes\delta_1 \cdots \delta_{i-1}\otimes 1)\otimes \delta_i\otimes 1)\delta_{i+1}\cdots\delta_s.
		\end{align}
		If $t=1$, $\phi_1(f_{k_{1}})(1\otimes v'\otimes\delta_1 \cdots \delta_{i-1}\otimes 1)$ is a path that ends with $\delta_{i-1}$, and in this case \eqref{eq term BBB} vanishes by Lemma \ref{lemma:phi=0-RBR}. If $t>1$, the terms in
		\[
		\phi_t(f_{j_{n+1-t}},\dots, f_{j_{n}})(1\otimes v'\otimes\delta_1 \cdots \delta_{i-1}\otimes 1)\\
		\]
		end with $\delta_{i-1}$, or are equal to 
		\[v_t(f_{j_{n+1-t}},\dots, f_{j_{n-2}})(1\otimes v' \otimes \delta_1 \cdots \delta_{i-2} \otimes \delta_{i-1}\otimes 1)\]
		which is zero by the inductive hypothesis.
	\end{proof}
	
	\begin{proposition}\label{proposition B}  Let $A=\kk Q/I$ be a gentle algebra. Assume $Q$ has no oriented cycles,  no parallel arrows, and contains a subquiver of the form
		\[ \xymatrix{ 
			1   \ar[r]_{\alpha_1} \ar@{.>}@/^0.5pc/[r]^{v'}  &   2 \ar[r]_{\alpha_2}  \ar@{.>}@/^0.7pc/[rr]^{u}       &   3  \ar[r]_{\alpha_3}   & 4
		} \]
		with $\omega=(\alpha_1,\alpha_2, \alpha_3)\in AP_3$, $v', u$ nonzero paths, and $v'u=0$.   If $l_n(f_1, \dots, f_n)(\omega)\neq 0$ for some  $n \geq 4$ and for some $f_i \in B^{1}(A)[1], 1 \leq i \leq n$ then 
		\begin{enumerate}[label=(\roman*),ref=(\roman*)] 
			\item
			$u=u' \alpha_3$, or
			\item \label{itemtwo}  
			$v'=\beta_1 \beta_2, u= \epsilon_1 \cdots \epsilon_p$ with $p >1$, $\epsilon_p \neq \alpha_3$, and there exists an arrow $\lambda$ parallel to $\beta_2 \epsilon_1$.  
		\end{enumerate}
	\end{proposition}
	
	\begin{proof}
		Assume $v'=  \beta_1 \cdots \beta_m$, $ u= \epsilon_1 \cdots \epsilon_p$, $\beta_m \epsilon_1=0$ and $\epsilon_p \neq  \alpha_3$. By assumption, $Q$ has no double arrows so $m>1$, and since $\epsilon_p \neq  \alpha_3$, the unique nonzero path from $1$ to $4$ is 
		$\alpha_1 u$.
		From \eqref{l_n} and Lemma \ref{lemma:phi=0-RBR} 
		we have that
		$l_n(f_1, \dots, f_n)(\omega)\neq 0$ implies that 
		\begin{align*}
			\phi_{n-1}&(f_1,\dots, \hat{f}_{i_1}, \dots,    f_n)(1\otimes v'\alpha_2 \otimes \alpha_3 \otimes 1) \\
			= & v_{n-1}(f_1,\dots, \hat{f}_{i_1}, \dots,    f_n)(1\otimes v' \otimes \alpha_2 \otimes \alpha_3 \otimes 1) \neq 0
		\end{align*} 
		for some $i_1$. By Remark \ref{rmk 4.6}, it follows that 
		\begin{align*}
			\phi_{n-2}(f_1,\dots, \hat{f}_{i_1}, \dots,  \hat{f}_{i_2}, \dots,  f_n)(1\otimes v' \otimes u \otimes 1) \neq 0
		\end{align*} for some $i_2$.
		This term equals
		\begin{align*}
			v_{n-2}&(f_1,\dots, \hat{f}_{i_1}, \dots, \hat{f}_{i_2}, \dots,  f_n)
			(1 \otimes \beta_1 \cdots \beta_{m-1} \otimes \beta_m \otimes \epsilon_1 \otimes 1)  \epsilon_2 \cdots \epsilon_p \\
			&- \sum_{i=2}^{p}v_{n-2}(f_1,\dots, \hat{f}_{i_1}, \dots, \hat{f}_{i_2}, \dots,  f_n) 
			(1 \otimes \beta_1 \cdots \beta_{m} \otimes  \epsilon_1  \cdots \epsilon_{i-1} \otimes \epsilon_i \otimes 1)\epsilon_{i+1}\cdots\epsilon_p.
		\end{align*}
		Using Lemmas \ref{lema s=2} and \ref{lema B}, we get that the unique possible nonzero summand is the first one. Moreover, the first summand also vanishes when $m\geq 3$, see Lemma \ref{lema A}.
		Hence $l_n(f_1, \dots, f_n)(\omega) =0$ for any $n\geq 4$ when $m\geq 3$.
		If $m=2$, using again Lemma \ref{lema A},
		we get that the first summand vanishes when $n\geq 5$, and its nonvanishing in the case $n=4$
		implies the existence of a subquiver as described in  \ref{itemtwo}. 
	\end{proof}
	
	\begin{remark}\label{rmk l_4 no 0}
		Let $A=\kk Q/I$ be a gentle algebra. Assume $Q$ has no oriented cycles, no parallel arrows, and it contains a subquiver of the form  described in  \ref{itemtwo}, that is, 
		
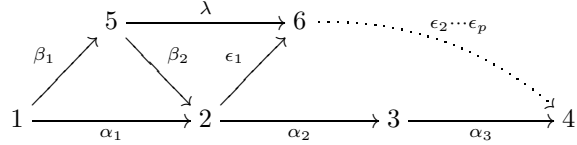
\[
		\xymatrix{
			& 5 \ar[rd]^{\beta_2} \ar[rr]^{\lambda}& & 6  \ar@{.>}@/^1pc/[rrrd]^{\epsilon_2\cdots\epsilon_p}& &\\
			1   \ar[rr]_{\alpha_1}\ar[ru]^{\beta_1} & &  2 \ar[rr]_{\alpha_2} \ar[ru]^{\epsilon_1}  & &    3  \ar[rr]_{\alpha_3} &  &4
		}	 
		\]
		with $\omega=(\alpha_1,\alpha_2, \alpha_3)\in AP_3$, $\epsilon_p\neq\alpha_3$, $\beta_1\beta_2\neq 0$, and $\epsilon_1\cdots\epsilon_p\neq 0$. Then, by the proof of Proposition \ref{proposition B} and Lemma \ref{lema A}, 
		$l_n(f_1, \ldots , f_n)(\omega)=0$ for $n\geq 5$, and $l_4(f_1, \ldots, f_4)(\omega)$ could be nonzero only  when there exist $i\neq j$ such that $\lambda$ appears as a summand in $f_i(\beta_2 \epsilon_1)$ and $f_j(\beta_1 \lambda)\epsilon_2 \cdots \epsilon_p$ is nonzero. Notice that this still holds if $\epsilon_p=\alpha_3$ as we will show in Theorem \ref{theo l_4 no 0 parte2}.
	\end{remark}

	\section{Main results} \label{sec 5}
	From Propositions \ref{prop 1}, \ref{proposition A}   and \ref{proposition B}, we conclude that if $A=\kk Q/I$ is a gentle algebra whose quiver $Q$ has no oriented cycles and no parallel arrows, and $l_n(f_1, \dots, f_n)(\omega)\neq 0$ for some  $n \geq 4$, for some $f_i \in B^{1}(A)[1], 1 \leq i \leq n$, and for some $\omega=(\alpha_1,\alpha_2, \alpha_3)\in AP_3$, then $Q$ contains a subquiver of the form
	\begin{enumerate}
		\item[(i)]
		\[ \xymatrix{ 
			1   \ar[r]_{\alpha_1}  &   2 \ar[r]_{\alpha_2}   \ar@{.>}@/^0.7pc/[r]^{u'}       &   3  \ar[r]_{\alpha_3}  \ar@{.>}@/^0.5pc/[r]^{w'}  & 4
		}\]
		with nonzero paths $u',w'$ such that $\alpha_1 u' \alpha_3 \neq 0$, 
		\item[(ii)]
		\[\xymatrix{ 
			1   \ar[r]_{\alpha_1} \ar@{.>}@/^0.5pc/[r]^{v'}  &   2 \ar[r]_{\alpha_2}  \ar@{.>}@/^0.7pc/[r]^{u'}       &   3  \ar[r]_{\alpha_3}   & 4
		} \]
		with nonzero paths $v',u'$  such that $\alpha_1 u' \alpha_3 \neq 0$, or
		\item[(iii)]
		
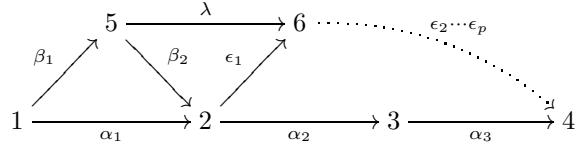
\[ \xymatrix{
			& 5 \ar[rd]^{\beta_2} \ar[rr]^{\lambda}& & 6  \ar@{.>}@/^1pc/[rrrd]^{\epsilon_2 \cdots \epsilon_p}& &\\
			1   \ar[rr]_{\alpha_1}\ar[ru]^{\beta_1} & &  2 \ar[rr]_{\alpha_2} \ar[ru]^{\epsilon_1}  & &    3  \ar[rr]_{\alpha_3} & & 4
		}	 \] 
		with nonzero paths 
		$\beta_1\beta_2 \alpha_2$ and  $\alpha_1 \epsilon_1 \cdots \epsilon_p, \epsilon_p \neq \alpha_3$.
	\end{enumerate}
	The following proposition  
	will allow us to find a nilpotence condition on a gentle algebra $A$ whose quiver has no parallel arrows and no oriented cycles. More precisely, we find conditions on the quiver $Q$ in order to get the vanishing of $l_n(f_1, \ldots, l_n)$ for all $n \geq 5$ and for all $f_i \in B^1(A)[1]$.

	\begin{proposition}
		Let $A=kQ/I$ be the  gentle algebra given by the quiver	
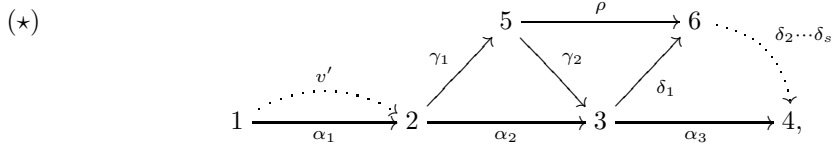
\[ \xymatrix{
			& & & 5 \ar[rd]^{\gamma_2} \ar[rr]^{\rho} & & 6 \ar@{.>}@/^1pc/[rd]^{\delta_2\cdots\delta_s} &\\
			1   \ar[rr]_{\alpha_1} \ar@{.>}@/^1pc/[rr]^{v'} &  &  2 \ar[rr]_{\alpha_2} \ar[ru]^{\gamma_1}  & &    3 \ar[ru]_{\delta_1} \ar[rr]_{\alpha_3} & & 4,
		}	\label{caso 2} \tag{$\star$} \]
		with $\omega=(\alpha_1,\alpha_2, \alpha_3)\in AP_3$, $v', \gamma_1 \gamma_2$ and $w'=\delta_1 \cdots \delta_s$ nonzero paths.
		Then there exists $f \in B(A)[1]$ such that  $l_n(f, \ldots, f)(\omega) \neq 0$ for any $n=2+3k, k \geq 1$.
	\end{proposition}

	\begin{proof}
		Let $v'=\beta_1\dots \beta_m$ and  $n\geq 4$. Consider   $$f=(\alpha_1\alpha_2||v'\alpha_2+\alpha_1\gamma_1\gamma_2)+(\alpha_2\alpha_3||\gamma_1\gamma_2\alpha_3+\alpha_2w')+(\gamma_2\delta_1||\rho) +(\gamma_1\rho||\alpha_2\delta_1).$$  
		Since $\alpha_1 \gamma_1 \gamma_2 \alpha_3$ is a nonzero path, by equation \eqref{l_n} and Lemma \ref{lemma:phi=0-RBR} we have
		\begin{align*} 
			l_n(f, \ldots, f)(\omega)=(-1)^n n[&\phi_{n-1}(f, \ldots, f) (1\otimes v'\alpha_2 \otimes \alpha_3\otimes 1) \\
			- &
			\phi_{n-1}(f, \ldots, f)(1\otimes \alpha_1\otimes \alpha_2w'\otimes 1)]
		\end{align*}
		which is, by definition, equal to
		\begin{align*} 
			&(-1)^n n [ v_{n-1}(f, \ldots, f)(1\otimes v'\otimes \alpha_2 \otimes \alpha_3\otimes 1)
			\\ 
			&+ v_{n-1}(f, \ldots, f)(1\otimes \alpha_1\otimes \alpha_2 \otimes \delta_1\otimes 1)\delta_2\cdots\delta_s \\   
			&+ \sum_{i=2}^s v_{n-1}(f, \ldots, f)(1\otimes \alpha_1\otimes \alpha_2 \delta_1\cdots \delta_{i-1}\otimes \delta_i\otimes 1)\delta_{i+1}\cdots \delta_s 
			].
		\end{align*}
		Since the only nonzero path  parallel  to $\alpha_1\alpha_2\delta_1\cdots \delta_{i-1}$ is $v'\alpha_2\delta_1\cdots \delta_{i-1}$, then $$v_{n-1}(f, \ldots, f)(1\otimes \alpha_1\otimes \alpha_2 \delta_1\cdots \delta_{i-1}\otimes \delta_i\otimes 1)\delta_{i+1}\cdots \delta_s=0$$ for all $i$. 
		Also, as $f(\beta_m\gamma_1)=0$, by Lemma \ref{lemma: no parallel} we get that
		\[v_{n-1}(f, \ldots, f)(1\otimes v'\otimes \alpha_2 \otimes \alpha_3\otimes 1)= (-1)^n \phi_{n-2}(f, \ldots, f)(1\otimes v' \otimes \gamma_1\gamma_2\alpha_3\otimes 1)=0.\]
		Hence, 
		\begin{align}
			l_n(f, \ldots, f)(\omega)
			=  (-1)^n n \ v_{n-1}(f, \ldots, f)(1\otimes \alpha_1\otimes \alpha_2 \otimes \delta_1\otimes 1)\delta_2\dots\delta_s  \label{top row}  \end{align}
		which is equal to 
		$-n(n-1)v_{n-2}(f, \ldots, f)(1\otimes \alpha_1\gamma_1\otimes \gamma_2\otimes \delta_1\otimes 1)\delta_2\cdots\delta_s$.
		This term vanishes if $n=4$ and, for $n \geq 5$, it is equal to
		\begin{align*}
			(-1)^nn(n-1)&(n-2) v_{n-3}(f, \ldots, f)(1\otimes \alpha_1\otimes \gamma_1\otimes \rho\otimes 1)\delta_2\cdots\delta_s
			\\
			=&
			\begin{cases}
				-\frac{n!}{(n-4)!}v_{n-4}(f, \ldots, f)(1\otimes \alpha_1\otimes \alpha_2 \otimes \delta_1\otimes 1)\delta_2\cdots\delta_s, & \mbox{ if $n\geq 6$,}\\
				5!\, v'\alpha_2 w', & \mbox{ if $n=5$.}
			\end{cases}   
		\end{align*}
		Now, using \eqref{top row}, we get that, for $n\geq 6$,
		\[l_n(f, \ldots, f) (\omega) = (-1)^{n} \frac{n!}{(n-3)!} l_{n-3}(f, \ldots, f)(\omega).\]
		Therefore, one can check that 
		\[l_n(f, \ldots, f)(\omega)= -(-1)^{\frac{k(k+1)}{2}} n! \,v'\alpha_2w'\neq 0 \] 
		for $n=2+3k$ and $k\geq1$.
	\end{proof}
	
	In the next two theorems we find a sufficient condition for the vanishing of $l_n(f_1, \ldots, f_n)$ for $n \geq 5$, and a necessary condition for the nonvanishing of $l_4(f_1, f_2, f_3, f_4)$, for $f_i \in B^1(A)[1]$. With this purpose,  in the next proof we will also pay attention to those terms involved that vanish also for $n=4$.
	
	\begin{theorem}\label{teo ln=0}
		Let $A=\kk Q/I$ be a gentle algebra. Assume $Q$ has no oriented cycles and no parallel arrows.
		If $Q$ does not contain subquivers of the form \eqref{caso 2}, then $l_n(f_1, \ldots, f_n)=0$ for all $n \geq 5$, for all $f_i \in B^1(A)[1], 1 \leq i \leq n$.
	\end{theorem}

	\begin{proof}
		Suppose that there exist $\omega=(\alpha_1, \alpha_2, \alpha_3)\in AP_3$ and  $f_1, \ldots, f_n\in B(A)[1]$ such that
		$l_n(f_1, \dots, f_n)(\omega)\neq 0$ and $n\geq 5$.
		It follows from Propositions \ref{prop 1}, 
		\ref{proposition A} and \ref{proposition B}, and Remark \ref{rmk l_4 no 0}
		that the quiver $Q$ has a subquiver of the form 
		\[ \xymatrix{
			& 1   \ar[r]_{\alpha_1} \ar@{.>}@/^1pc/[r]^{v'} &   2 \ar[r]_{\alpha_2} \ar@{.>}@/^1pc/[r]^{u'}   &  \ar@{.>}@/^1pc/[r]^{w'}  3  \ar[r]_{\alpha_3} & 4
		}	 \]
		where the nonzero paths $v'$ and $w'$ may not appear simultaneously.  
		Set $v'=\beta_1\cdots\beta_m$, $u'=\gamma_1\cdots\gamma_r$, and $w'=\delta_1\cdots\delta_s$, {where $m,r,s >1$}. 	By definition, $l_n(f_1, \dots, f_n)(\omega)$ is a linear combination of terms of the form
		\begin{align}\label{eqn1}
			\phi_{n-1}(f_1, \dots, \hat{f}_i,\dots, f_n)(1\otimes v'\alpha_2\otimes \alpha_3\otimes 1), \mbox{ and }\\ \label{eqn1'}
			\phi_{n-1}(f_1, \dots, \hat{f}_i,\dots, f_n)(1\otimes\alpha_1\otimes \alpha_2w'\otimes 1).
		\end{align}
		We  first consider the term \eqref{eqn1}, which is equal to 
		\[
		v_{n-1}(f_1, \dots, \hat{f}_i,\dots, f_n)(1\otimes v'\otimes \alpha_2\otimes \alpha_3\otimes 1), \]
		and can only be written as a linear combination of  terms of the form
		\begin{align}  \label{termino 2}
			& v_{n-2}(f_1, \dots, \hat{f}_i,\dots, \hat{f}_k,\dots, f_n)(1\otimes v'\otimes \gamma_1\cdots \gamma_{j-1}\otimes \gamma_j\otimes 1)\gamma_{j+1}\cdots\gamma_r\alpha_3, \\ \label{termino 3}
			& v_{n-2}(f_1, \dots, \hat{f}_i,\dots, \hat{f}_k,\dots, f_n)(1\otimes v'\otimes u'\otimes\alpha_3\otimes 1), \mbox{ and } \\
			\label{termino 1}
			& v_{n-2}(f_1, \dots, \hat{f}_i,\dots, \hat{f}_k,\dots, f_n)(1\otimes\beta_1\cdots \beta_{m-1}\otimes \beta_m\otimes \gamma_1\otimes 1)\gamma_2\cdots \gamma_r\alpha_3.
		\end{align}
		Since $Q$ has no oriented cycles, the only nonzero path parallel to $v'\gamma_1\cdots\gamma_{j-1}$ is $\alpha_1\gamma_1\cdots \gamma_{j-1}$,
		then \eqref{termino 2} vanishes
		by Lemma \ref{lemma:phi=0-RBR} for any $n\geq 4$. 
		By definition, \eqref{termino 3} is a linear combination of terms of the form
		\begin{align*}
			\phi_{n-3}(f_{i_1}, \dots,f_{i_{n-3}} )&(1\otimes \beta_1\cdots\beta_{m-1} f_{i_{n-2}}(\beta_m \gamma_1)\gamma_2\cdots \gamma_r\otimes \alpha_3\otimes 1),  \mbox{ and }\\
			\phi_{n-2-t}(f_{k_1}, \dots,f_{k_{n-2-t}})&(1\otimes \phi_t(f_{k_{n-1-t}}, \dots,f_{k_{n-2}})(1\otimes v' \otimes u' \otimes 1)\otimes \alpha_3\otimes 1), 
		\end{align*}
		for $2\leq t < n-2$. Since $\gamma_r\alpha_3  \neq 0$, and 
		\begin{align*}
			\phi_t(f_{k_{n-1-t}},& \dots,f_{k_{n-2}})(1\otimes v' \otimes u' \otimes 1)\\
			= &v_t(f_{k_{n-1-t}}, \dots,f_{k_{n-2}})(1\otimes \beta_1\cdots\beta_{m-1}\otimes \beta_m\otimes \gamma_1\otimes 1)\gamma_2\cdots\gamma_r \\
			& - \sum_{j=2}^r v_t(f_{k_{n-1-t}}, \dots,f_{k_{n-2}})(1\otimes v' \otimes  \gamma_1\cdots \gamma_{j-1}\otimes \gamma_j\otimes 1)\gamma_{j+1}\cdots\gamma_r,
		\end{align*}
		we have that \eqref{termino 3} is zero for $n=4$ and the only nonzero terms which can appear in \eqref{termino 3} for $n\geq 5$ are
		\begin{align*}
			\phi_{n-2-t}(f_{k_1}, \dots,f_{k_{n-2-t}})&(1\otimes v_t(f_{k_{n-1-t}}, \dots,f_{k_{n-2}})(1\otimes v' \otimes  \gamma_1\cdots \gamma_{r-1}\otimes \gamma_r\otimes 1) \otimes \alpha_3\otimes 1).
		\end{align*}
		By definition of $v_t$, we have to study the terms of the form 
		\begin{align*}
			\phi_{t-c}(f_{p_1}, \dots,f_{p_{t-c}})(1\otimes \phi_c(f_{p_{t-c+1}}, \dots,f_{p_{t}})(1\otimes v'\otimes \gamma_1\cdots\gamma_{r-1} \otimes 1 ) \otimes \gamma_r\otimes 1) 
		\end{align*}
		for all $1 \leq c <t$. Since the only nonzero path parallel to $v' \gamma_1\cdots\gamma_{r-1}$ is $\alpha_1 \gamma_1\cdots\gamma_{r-1}$ and $\gamma_{r-1}\gamma_r\neq 0$, we get that \eqref{termino 3} vanishes for all $n \geq 4$.
		
		Now we will see that \eqref{termino 1} equals zero. 
		Since $\phi_t(1 \otimes \beta_m \otimes \gamma_1 \otimes 1)=0$ when $t>1$, we have that \eqref{termino 1} is a linear combination of terms of the form
		\begin{equation*}
			\phi_{n-3}(f_{q_1}, \dots,f_{q_{n-3}})(1\otimes \beta_1\cdots \beta_{m-1}\otimes f_{q_{n-2}}(\beta_m \gamma_1)\otimes 1)\gamma_2\cdots\gamma_r \alpha_3.
		\end{equation*}
		These terms vanish for all $n \geq 5$, unless there exists a nonzero path
		$\lambda_1\cdots \lambda_h$  parallel to $\beta_m\gamma_1$: 
		
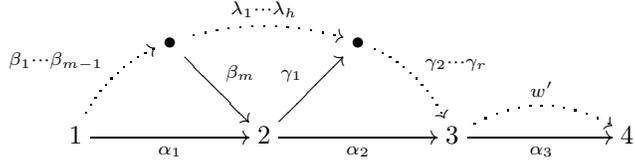
\[ \xymatrix{ 
			&   \bullet  \ar[rd]^{\beta_m} \ar@{.>}@/^.5pc/[rr]^{\lambda_1\cdots\lambda_h} &  & \bullet \ar@{.>}@/^.5pc/[rd]^{\gamma_2\cdots\gamma_r} & &\\
			1   \ar@{.>}@/^.5pc/[ru]^{\beta_1\cdots\beta_{m-1}} \ar[rr]_{\alpha_1} & &  2 \ar[rr]_{\alpha_2} \ar[ru]^{\gamma_1}   &  &  3   \ar@{.>}@/^1pc/[rr]^{w'} \ar[rr]_{\alpha_3} &  & 4.
		}
		\]
		Notice that $\lambda_h\neq \gamma_1$ since $A$ is gentle without oriented cycles, and hence $\lambda_h \gamma_2 =0$. In this case, it vanishes if $h=1$, $m=2$, and $n\geq 5$,
		and otherwise it can only be a linear combination of terms of the form
		\begin{align*}
			v_{n-3}&(f_{q_1}, \dots,f_{q_{n-3}})(1\otimes \beta_1\cdots \beta_{m-1}\otimes \lambda_1\cdots \lambda_{h-1}\otimes \lambda_h \otimes 1)\gamma_2\cdots\gamma_r \alpha_3, & \mbox{if $h>1$,} \\
			v_{n-3}&(f_{q_1}, \dots,f_{q_{n-3}})(1\otimes \beta_1\cdots \beta_{m-2}\otimes \beta_{m-1}\otimes \lambda_1\otimes 1)\gamma_2\cdots\gamma_r \alpha_3, & \mbox{ if $h=1$}.
		\end{align*}
		The first term vanishes since there are no nonzero paths parallel to $\beta_1\cdots \beta_{m-1} \lambda_1\cdots \lambda_{h-1}$ since $Q$ has no oriented cycles,   thus \eqref{termino 1} vanishes for $h>1$. The second term, for 
		$h=1$ and $m>2$, vanishes if there is no nonzero path $y$ parallel to $\beta_{m-1}\lambda_1$ such that $\beta_{m-2} y=0$. Since $Q$ has no oriented cycles, the unique possibility for such $y$ is $y=z \lambda_1$, for $z$ a nonzero path parallel to $\beta_{m-1}$. 
		Arguing as in Example \ref{ejemplo s=1},
		one can check that the only arguments appearing in the next steps of the computation
		are $1\otimes \beta_1\cdots \beta_{m-1}\otimes \lambda_1 \otimes 1$ and $1\otimes \beta_1\cdots \beta_{m-2}\otimes z \lambda_1 \otimes 1$. Applying Remark \ref{loop argument} we get that \eqref{termino 1} vanishes. 
		Therefore, \eqref{eqn1} vanishes for $n\geq 5$.
		
		It remains to show that \eqref{eqn1'} vanishes. By definition, it is equal to
		\begin{align}\label{eq v n-1 a}
			-v_{n-1}&(f_1, \dots, \hat{f}_i,\dots, f_n)(1\otimes\alpha_1\otimes \alpha_2\otimes \delta_1\otimes 1)\delta_2\cdots \delta_s\\ \notag
			&-\sum_{j=2}^s v_{n-1}(f_1, \dots, \hat{f}_i,\dots, f_n)(1\otimes\alpha_1\otimes \alpha_2\delta_1\cdots \delta_{j-1}\otimes \delta_j\otimes 1)\delta_{j+1}\cdots \delta_s.
		\end{align}
		By Lemma \ref{lemma:phi=0-RBR}, all the last terms vanish for all $n \geq 3$ since, as $Q$ has no oriented cycles, $v'\alpha_2\delta_1\cdots\delta_{j-1}$ is the only nonzero path parallel to $\alpha_1\alpha_2\delta_1\cdots\delta_{j-1} $. 
		The term \eqref{eq v n-1 a} can only be a linear combination of terms of the form
		\begin{align*}\phi_{n-2}&(f_1, \dots, \hat{f}_i, \dots, \hat{f}_j, \dots, f_n)(1\otimes\alpha_1\gamma_1\cdots\gamma_r\otimes \delta_1\otimes 1)\delta_2\cdots \delta_s\\
			&= v_{n-2}(f_1, \dots, \hat{f}_i, \dots, \hat{f}_j, \dots, f_n)(1\otimes\alpha_1\gamma_1\cdots\gamma_{r-1}\otimes \gamma_r\otimes  \delta_1\otimes 1)\delta_2\cdots \delta_s.
		\end{align*}
		Let $\rho_1\cdots\rho_c$ be a nonzero path parallel to $\gamma_r\delta_1$: 
		
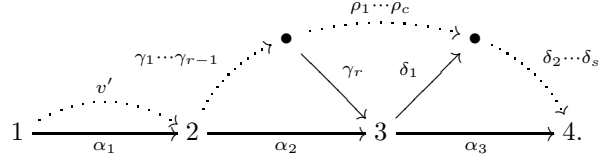
\[ \xymatrix{ 
			& &  & \bullet  \ar[rd]^{\gamma_r} \ar@{.>}@/^.5pc/[rr]^{\rho_1\cdots\rho_c} &  & \bullet \ar@{.>}@/^.5pc/[rd]^{\delta_2\cdots\delta_s}  &\\
			1    \ar[rr]_{\alpha_1} \ar@{.>}@/^1pc/[rr]^{v'} & &  2 \ar[rr]_{\alpha_2} \ar@{.>}@/^.5pc/[ru]^{\gamma_1\cdots\gamma_{r-1}}   &  &  3 \ar[ur]^{\delta_1}   \ar[rr]_{\alpha_3} & & 4.}
		\]
		Thus, the only terms which can appear in \eqref{eq v n-1 a} are multiples of
		\begin{equation} \label{eq c}\phi_{n-3}(f_{t_1}, \dots, f_{t_{n-3}})(1\otimes\alpha_1\gamma_1\cdots\gamma_{r-1}\otimes \rho_1\cdots \rho_c\otimes 1)\delta_2\cdots \delta_s\end{equation}
		for $n \geq 5$.
		If $c>1$, using that $\rho_c \delta_2=0$, we get that \eqref{eq c} is equal to
		\begin{align*}
			-v_{n-3}(f_{t_1}, \dots, f_{t_{n-3}})(1\otimes\alpha_1\gamma_1\cdots\gamma_{r-1}\otimes \rho_1\cdots \rho_{c-1}\otimes \rho_c\otimes 1)\delta_2\cdots \delta_s.
		\end{align*}
		This is zero since there are no nonzero paths parallel to  $\alpha_1\gamma_1\cdots\gamma_{r-1}\rho_1\cdots\rho_{c-1}$.
		{If $c=1$, \eqref{eq c} is equal to }
		\begin{align*}  
			& v_{n-3}(f_{t_1}, \dots, f_{t_{n-3}})(1\otimes\alpha_1\gamma_1\cdots\gamma_{r-2}\otimes \gamma_{r-1} \otimes \rho_1\otimes 1)\delta_2\cdots \delta_s.
		\end{align*}
		Notice that for $c=1$, we have that $r>2$ since by assumption the quiver $Q$ does not contain a subquiver of the form  \eqref{caso 2}.
		In this case, this term vanishes if there is no nonzero path $y$ parallel to $\gamma_{r-1}\rho_1$ such that $\gamma_{r-2} y=0$.
		Since $Q$ has no oriented cycles, the unique possibility for such $y$ is $y = z\rho_1$,
		for $z$ a nonzero path parallel to $\gamma_{r-1}$. Arguing as in Example \ref{ejemplo s=1},
		one can check  that the  only arguments appearing in the next steps of the computation
		are $1\otimes\alpha_1\gamma_1\cdots\gamma_{r-1}\otimes \rho_1 \otimes 1$ and $1\otimes \alpha_1\gamma_1\cdots\gamma_{r-2}\otimes z\rho_1 \otimes 1$. Applying Remark \ref{loop argument} we get that \eqref{eq c} vanishes for all $n \geq 5$. Therefore, \eqref{eqn1'} vanishes and the proof is completed.
	\end{proof}

	\begin{theorem}\label{theo l_4 no 0 parte2}
		Let $A=\kk Q/I$ be a gentle algebra. Assume $Q$ has no oriented cycles and no parallel arrows.
		If $l_4(f_1, f_2, f_3, f_4) \neq 0$ for some $f_i \in B^1(A)[1]$,
		then $Q$ contains a subquiver of the form
		\begin{align}\label{diamond} \tag{$\diamond$} 
			
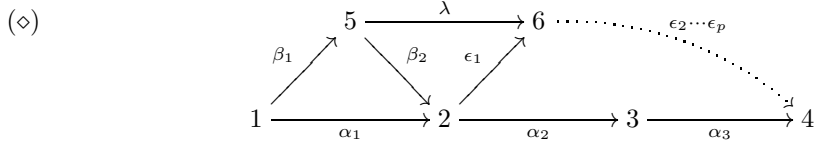
\xymatrix{
				& 5 \ar[rd]^{\beta_2} \ar[rr]^{\lambda}& & 6  \ar@{.>}@/^1pc/[rrrd]^{\epsilon_2\cdots\epsilon_p}& &\\
				1   \ar[rr]_{\alpha_1}\ar[ru]^{\beta_1} & &  2 \ar[rr]_{\alpha_2} \ar[ru]^{\epsilon_1}  & &    3  \ar[rr]_{\alpha_3} &  & 4
			}	
		\end{align} 
		with $\omega=(\alpha_1,\alpha_2, \alpha_3)\in AP_3$, $\beta_1\beta_2\neq 0$, $\epsilon_1\cdots\epsilon_p\neq 0$,
		and $l_4(f_1,f_2,f_3,f_4)(\omega)\neq 0$ for some $f_i\in B^1(A)[1], , 1 \leq i \leq 4$.
	\end{theorem}
	\begin{proof}
		From Propositions \ref{prop 1}, \ref{proposition A} and \ref{proposition B}, we know that $Q$ contains a subquiver of the form \eqref{diamond}
		with $\epsilon_p\neq\alpha_3$, or a subquiver of the form
		
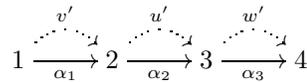
\[ \xymatrix{
			& 1   \ar[r]_{\alpha_1} \ar@{.>}@/^1pc/[r]^{v'} &   2 \ar[r]_{\alpha_2} \ar@{.>}@/^1pc/[r]^{u'}   &  \ar@{.>}@/^1pc/[r]^{w'}  3  \ar[r]_{\alpha_3} & 4
		}	 \]
		where the nonzero paths $v'$ or $w'$ may not appear simultaneously. 
		Assume we are in the second case, with the same notation as in the proof of Theorem \ref{teo ln=0}. 
		So, $l_4(f_1,f_2,f_3,f_4)(\omega)$ is a linear combination of terms of the form \eqref{termino 1} and \eqref{eq c}:
		\begin{align*}
			& v_{2}(f_{i_1}, f_{i_2})(1\otimes\beta_1\cdots \beta_{m-1}\otimes \beta_m\otimes \gamma_1\otimes 1)\gamma_2\cdots \gamma_r\alpha_3, \mbox{ and}  \\
			&\phi_1(f_{t})(1\otimes\alpha_1\gamma_1\cdots\gamma_{r-1}\otimes \rho_1\cdots \rho_c\otimes 1)\delta_2\cdots \delta_s.
		\end{align*}
		The second one is zero since, as $Q$ has no oriented cycles, there is no nonzero path starting with $\alpha_1$ and ending with $\delta_s$.
		The first one vanishes when $m>2$ since there is no nonzero path starting with $\beta_1$ and ending with $\alpha_3$. Assume $m=2$. In this case, 
		the nonvanishing of the term
		\[ \phi_1(f_{j})(1\otimes\beta_1\otimes f_k(\beta_2 \gamma_1)\otimes 1)\gamma_2\cdots \gamma_r\alpha_3\]
		implies the existence of a nonzero path $\lambda_1 \cdots \lambda_h$ parallel to $\beta_2 \gamma_1$, with $\lambda_1 \neq \beta_2$. Since $Q$ has no oriented cycles, this means that $\lambda_h \neq \gamma_1$. Therefore $\lambda_h \gamma_2=0$, and hence the last term can be nonzero only if $h=1$.
		In this case we have a subquiver of the form \eqref{diamond} (with $\epsilon_1 \cdots \epsilon_p=\gamma_1 \cdots \gamma_r \alpha_3$). 
		
		Finally, if $Q$ has a subquiver of the form \eqref{diamond} then, for instance, for
		\begin{align*}
			& f_1=(\alpha_1 \alpha_2|| \beta_1 \beta_2 \alpha_2), & & f_2=(\alpha_2 \alpha_3 || \epsilon_1 \cdots \epsilon_p),  \\ & f_3=(  \beta_2\epsilon_1|| \lambda), \mbox{ and } &  &f_4=(\beta_1 \lambda ||\alpha_1 \epsilon_1),
		\end{align*}
		we have that $l_4(f_1, f_2, f_3, f_4)(\omega) = -\alpha_1 \epsilon_1 \cdots \epsilon_p\neq 0$.
	\end{proof}

	\begin{theorem}\label{teo MC elements}
		Let $A=\kk Q/I$ be a gentle algebra. Assume $Q$ has no oriented cycles, no parallel arrows,
		and no subquivers of the form \eqref{caso 2}.
		Then  $f= \sum_{i=1}^m f_i t^i$ is a Maurer-Cartan element if and only if  $f_i$ is a $2$-cocycle for all $i$.
	\end{theorem}
	
	\begin{proof} 
		From Lemmas \ref{l_2 =0} and \ref{l_3 =0}, and Theorem \ref{teo ln=0}, we have that $l_n(f_{i_1}, \ldots, f_{i_n})\neq 0$ only for $n=1,4$. Therefore,
		$f= \sum_{i=1}^m f_i t^i$ is a Maurer-Cartan element if and only if it satisfies the equations \eqref{MC}
		\[\delta^2(f_k)+ {\frac{1}{4!}} \sum_{i_1+i_2+i_3+i_4=k}l_4(f_{i_1}, f_{i_2}, f_{i_3}, f_{i_4})=0\]
		for all $k\geq 1$. If $l_4=0$, the result is clear.
		
		Assume $l_4 (f_{i_1}, f_{i_2}, f_{i_3}, f_{i_4}) \neq 0$. By Theorem \ref{theo l_4 no 0 parte2}, we can affirm that
		$Q$ has a subquiver of the form \eqref{diamond}. 
		We will prove first that if $g_j$ is a $2$-cocycle for all $j$ with $1 \leq j \leq 4$, then $l_4(g_1,g_2,g_3,g_4) = 0$. In fact, 
		$l_4(g_1,g_2,g_3,g_4)\neq 0$ implies that  
		$l_4(g_1,g_2,g_3,g_4)(\omega) \neq 0$ for $\omega=(\alpha_1, \alpha_2, \alpha_3)$ in the subquiver of the form \eqref{diamond}, see Theorem \ref{theo l_4 no 0 parte2}.
		Moreover, by Remark \ref{rmk l_4 no 0}, 
		$g_j(\beta_1\lambda )\delta_2$ is nonzero, for some $j$, and hence $g_j$ is not a $2$-cocycle since 
		\begin{align*}
			\delta^2(g_j)(\beta_1,\lambda,\delta_2)= - g_j(\beta_1\lambda)\delta_2 + \beta_1 g_j(\lambda \delta_2) 
		\end{align*}
		and $g_j(\beta_1\lambda ) \delta_2$ is not a path starting with $\beta_1$ in  \eqref{diamond}.
		
		Now it is clear that if all $f_i$ are $2$-cocycles, then $\sum_{i=1}^mf_i$ is a Maurer-Cartan element. For the converse, suppose that not all $f_i$ are $2$-cocycles, and  let $t$ be such that $f_t$ is not a $2$-cocycle and $f_j$ are $2$-cocycles for all $j$ with $1 \leq j<t$. 
		As we noted above, 
		$l_4(f_{i_1}, f_{i_2}, f_{i_3}, f_{i_4})= 0$
		when $i_1 + i_2 + i_3 + i_4=t$.
		Hence the Maurer-Cartan equation for $k=t$ reduces to  $\delta^2(f_t)=0$, which is a contradiction. 
	\end{proof}
	
	The previous result is not true if the quiver has parallel arrows or oriented cycles, as we can see in the following examples.
	
	\begin{example}
		Let $A=\kk Q/I$ be the algebra with quiver
		
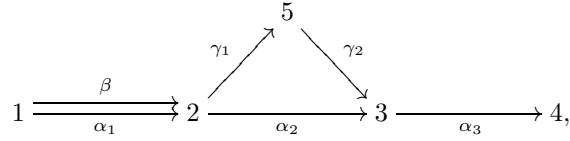
\[ \xymatrix{
			& & & 5  \ar[rd]^{\gamma_2}& &\\
			1  \ar@<1ex>[rr]^{\beta} \ar[rr]_{\alpha_1} & &  2 \ar[rr]_{\alpha_2} \ar[ru]^{\gamma_1}  & &    3  \ar[rr]_{\alpha_3} &  & 4,
		}	 \]
		and  $I=\langle\alpha_1\alpha_2, \alpha_2\alpha_3, \beta\gamma_1\rangle$.
		Let $f_1=(\alpha_1\alpha_2||\beta\alpha_2 +\alpha_1\gamma_1\gamma_2) + (\alpha_2\alpha_3||\gamma_1\gamma_2\alpha_3) + (\beta\gamma_1||\alpha_1\gamma_1)$ and $f_3= (\alpha_1\alpha_2||\alpha_1\gamma_1\gamma_2)$.
		One can check that $l_n(g_1,\dots,g_n)(\alpha_1,\alpha_2,\alpha_3)=0$ for all $g_j\in B^1(A)[1]$ and for all $n\neq 1,3$. Moreover, $f_1$ is a $2$-cocycle and $l_3(f_{i_1}, f_{i_2},f_{i_3})(\alpha_1,\alpha_2,\alpha_3)=0$ except for
		$$l_3(f_1, f_1, f_1)(\alpha_1,\alpha_2,\alpha_3)= 6 \alpha_1\gamma_1\gamma_2\alpha_3 \neq 0.$$ 
		As $l_1(f_3)(\alpha_1,\alpha_2,\alpha_3)= \alpha_1\gamma_1\gamma_2\alpha_3 \neq 0$
		we have that $f_3$ is not a $2$-cocycle, and  $\displaystyle l_1(f_3) - \frac{1}{6}l_3(f_1, f_1, f_1)=0$, hence $f=f_1 t+f_3 t^3$ is a Maurer-Cartan element.
	\end{example}

	\begin{example}
		Let $A=\kk Q/I$ be the algebra with quiver
		
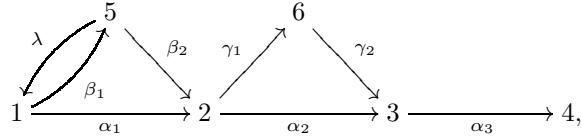
\[ \xymatrix{
			& 5\ar[rd]^{\beta_2} \ar@/_0.5pc/[ld]_{\lambda} & & 6  \ar[rd]^{\gamma_2}& &\\
			1  \ar@/_0.5pc/[ru]_{\beta_1} \ar[rr]_{\alpha_1} &  &  2 \ar[rr]_{\alpha_2} \ar[ru]^{\gamma_1}  & &    3  \ar[rr]_{\alpha_3} & & 4,
		}	 \]
		and  $I=\langle\alpha_1\alpha_2, \alpha_2\alpha_3, \beta_2\gamma_1, \beta_1\lambda, \lambda\beta_1 \rangle$.
		Let $f_1=(\alpha_1\alpha_2||\beta_1\beta_2\alpha_2 +\alpha_1\gamma_1\gamma_2) + (\alpha_2\alpha_3||\gamma_1\gamma_2\alpha_3) + (\beta_2\gamma_1||\lambda\alpha_1\gamma_1) + (\beta_1\lambda|| e_1) + (\lambda\beta_1|| e_5)$ and $f_4= -(\alpha_1\alpha_2||\alpha_1\gamma_1\gamma_2)$.
		One can check that $l_n(g_1,\dots,g_n)(\omega)=0$ for all $g_j\in B^1(A)[1]$,  for all $\omega\in AP_3$, and for all $n\neq 1,4$. Moreover, $f_1$ is a $2$-cocycle and $l_4(f_{i_1}, f_{i_2},f_{i_3}, f_{i_4})(\omega)=0$ except for
		$$l_4(f_1, f_1, f_1, f_1)(\alpha_1,\alpha_2,\alpha_3)=-24 \alpha_1\gamma_1\gamma_2\alpha_3 \neq 0.$$
		Since $l_1(f_4)(\alpha_1,\alpha_2,\alpha_3)= -\alpha_1\gamma_1\gamma_2\alpha_3 \neq 0$ we have that $f_4$ is not a $2$-cocycle 
		and $f=f_1 t+f_4 t^4$ is a Maurer-Cartan element because $\displaystyle l_1(f_4) - \frac{1}{24}l_4(f_1, f_1, f_1, f_1)=0$.
	\end{example}
	
	\section*{Acknowledgements}
	This work began at the Women in Noncommutative
	Algebra and Representation Theory (WINART3) workshop, held at the Banff International Research Station (BIRS) in April 2022. The authors would like to thank the organizers of WINART3 for this collaboration opportunity.
	M. M\"{u}ller gratefully acknowledges the support of Association for Woman in Mathematics (AWM), through AWM Mathematical Endeavors Revitalization Program (MERP) grant.

%
%

\end{document}